\documentclass[12pt,a4paper,reqno]{amsart}
\usepackage{amssymb}
\usepackage{amscd}
\usepackage[matrix,arrow,curve,cmtip]{xy}
\input xy
\xyoption{all}
\numberwithin{equation}{section}

\newcommand{\intl}{[ \hspace{-0.5mm} [}
\newcommand{\intr}{] \hspace{-0.5mm} ]}

\theoremstyle{plain}

\newtheorem{theorem}[subsection]{Theorem}
\newtheorem{proposition}[subsection]{Proposition}
\newtheorem{lemma}[subsection]{Lemma}
\newtheorem{corollary}[subsection]{Corollary}
\newtheorem{conjecture}[subsection]{Conjecture}

\theoremstyle{definition}

\newtheorem{definition}[subsection]{Definition}
\newtheorem{remark}[subsection]{Remark}

\newtheorem{notation}[subsection]{Notation}

\begin{document}

\title[New bases of some Hecke algebras]{New bases of some Hecke algebras via Soergel bimodules}

\author{Nicolas Libedinsky}
\address{IMJ, Universit\'e Paris-Diderot}
\email{libedinsky@math.jussieu.fr}

\begin{abstract}  For extra-large Coxeter systems ($m(s,r)>3$), we construct a natural and explicit set of Soergel bimodules $D=\{D_w\}_{w\in W}$ such that each $D_w$ contains as a direct summand (or is equal to) the indecomposable Soergel bimodule $B_w$. When decategorified, we prove that $D$ gives rise to a set $\{d_w\}_{w\in W} $ that is actually a basis of the Hecke algebra. This basis is close to the Kazhdan-Lusztig basis and  satisfies a ``positivity condition''. 
\end{abstract}

\maketitle
\today

\section{Introduction}

The understanding of indecomposable Soergel bimodules is a central problem in representation theory,  combinatorics and  knot theory.

 The study of the indecomposable Soergel bimodules is the main algebraic and combinatorial way to attack character formulae problems for algebraic groups : either to prove the Lusztig conjecture (\cite{Lu}) in characteristic bigger than the Coxeter number or even to conjecture what happens in small characteristic.

This study also has a  close relation with the positivity of the coefficients of Kazhdan-Lusztig  polynomials (\cite{KL1}) via Soergel's conjecture, that gives  a conjectural interpretation of that coefficients. This indecomposable bimodules also play an important role in calculating Khovanov-Rozansky homology of links (\cite{K}).

One might see this paper as a first step in the program of explicitly constructing these elusive indecomposable bimodules. We expect that with a refinement of our methods one might find all the indecomposable bimodules, at least for extra-large Coxeter groups.

 Let us be more precise. To any Coxeter system $(W,\mathcal{S})$, we can associate the tensor category \textbf{B} of Soergel bimodules. This is a category of $\mathbb{Z}$-graded bimodules over a polynomial ring $R$ upon which $W$ acts. There exists an isomorphism of rings $\varepsilon :
\mathcal{H}\rightarrow \langle \mathbf{B}\rangle$, where $\mathcal{H}$ is the Hecke algebra of $(W,\mathcal{S})$ and $\langle \mathbf{B}\rangle$ is the split Grothendieck group of \textbf{B}. This map may be used to view \textbf{B} as a categorification of the Hecke algebra.

The indecomposable objects $B_x$ of $ \mathbf{B}$ are parametrized (up to isomorphism and shifts in gradings) by the elements $x\in W$ of the Coxeter group. If $\bar{s}=(s_1,\ldots,s_n)$ is a reduced expression of $x\in W$, then $B_x$ is a direct summand of the bimodule $\theta_{\bar{s}}=\theta_{s_1}\theta_{s_2}\cdots \theta_{s_n}\in \mathbf{B}$, where $\theta_s=R\otimes_{R^s}R$ and $R^s$ is the subspace of $R$ fixed by $s\in \mathcal{S}$.

 For $s,r\in \mathcal{S}$,  we have introduced in the paper \cite{L}, a morphism $f_{sr}$ :  it is  the unique (up to a scalar) degree zero morphism in the space $\mathrm{Hom}_{\mathbf{B}}(\underbrace{\theta_s\theta_r\theta_s\cdots}_{m(s,r)},\underbrace{\theta_r\theta_s\theta_r\cdots}_{m(s,r)})$.  At this point, the following question is quite natural : what happens if we apply ``all possible'' morphisms of the form $\mathrm{id}\otimes f_{sr}\otimes \mathrm{id}$  to $\theta_{\bar{s}}$ ? Quite surprisingly we find that, at least when the Coxeter group is extra-large, the image of this composition of idempotents is a well-defined Soergel bimodule. 

More precisely we consider the graph $\mathcal{R}(x)$ with vertices reduced expressions for $x$, and edges corresponding to braids relations. To each circuit $c$ in $\mathcal{R}(x)$ with starting point $\bar{s}$ which visits every vertex we associate a natural idempotent  $f^e_{\bar{s},c}$ given by composing morphisms of the form $\mathrm{id}\otimes f_{sr}\otimes \mathrm{id}$ along  $c$. We prove that the image
$E_x=$Im($f^e_{\bar{s},c}$) does not depend (up to isomorphism) on the choice of the reduced expression $\bar{s}$, or on the choice of the path $c$.

Although $B_x$ is a direct summand of $E_x$, these two bimodules will rarely be isomorphic. The problem is that we have to examine more  morphisms than those of the form $f_{sr}$. Let us define, for $s,r\in\mathcal{S}$ and $1\leq n \leq m(s,r)$ the element $sr(n)=\underbrace{srs\cdots}_{n}$. We can describe $f_{rs}\circ f_{sr}$ as the unique degree zero idempotent that factors through the indecomposable $B_{sr(m)}$, with $m=m(s,r).$ So, by analogy we define  $f_{sr}^2(n)$ as the unique  idempotent in  End($\underbrace{\theta_s\theta_r\theta_s\cdots}_{n})$  that factors through $B_{sr(n)}$. We remark that $f_{sr}^2(n)$ is not defined over all fields because there are denominators involved in its definition, and this is coherent with the fact that Soergel's conjecture is not true in all characteristics. 

 If $n<m(s,r)$  the space $A_n(s,r)$ of applications from $\underbrace{\theta_s\theta_r\theta_s\cdots}_{n}$ to $\underbrace{\theta_r\theta_s\theta_r\cdots}_{n},$ does not contain a degree zero bimodule morphism. But in another Soergel category $\mathbf{B'}$ where $m(s,r)=n$ there exists a well-defined degree zero morphism $f_{sr}$ that we call $f_{sr}^{\mathbf{B'}}$ ; if we repeat formally the formulas defining $f_{sr}^{\mathbf{B'}}$ we obtain a well-defined map (not a morphism of bimodules) in $A_n(s,r)$. We have in the same way a well-defined degree zero map in $A_n(r,s)$. We believe that the composition of these two ``not morphisms'' is the honest morphism $f_{sr}^2(n)$, and this may provide a way to calculate this morphism.

Now our philosophy is to apply to $\theta_{\bar{s}}$ the $f_{sr}$ morphisms all the times it is possible to apply it, and even \textit{when it is impossible}. By this we mean that, as well as applying the morphisms $f_{sr}$, we also apply the maps $f_{sr}^{\mathbf{B'}}$ for all categories $\mathbf{B'}.$

More rigorously the question is the following :  what happens if we apply  ``all possible'' morphisms of the form $\mathrm{id}\otimes f_{sr}\otimes \mathrm{id}$  and morphisms of the form $\mathrm{id}\otimes f_{sr}^2(n) \otimes \mathrm{id}$ with $n>3$ to $\theta_{\bar{s}}$ ?

The answer is the same as before : in the extra-large case we obtain an idempotent $f^d_{\bar{s},c}$ of End($\theta_{\bar{s}}$) whose image is $D_x=$Im($f^d_{\bar{s},c}$), a bimodule depending only on $x$.

We prove that $B_x$ is a direct summand of $D_x$, who is itself a direct summand of $E_x$, so we have a chain of bimodules in Soergel category : $B_x \subseteq D_x \subseteq E_x\in \mathbf{B}.$ We prove that the two sets $\{\eta(\langle D_x\rangle)\}_{x\in W}$ and $\{\eta(\langle E_x\rangle)\}_{x\in W}$ are bases of the Hecke algebra, where $\eta :\langle \mathbf{B}\rangle
\rightarrow \mathcal{H}$ is by definition the inverse morphism of $\varepsilon$. By general arguments due to Soergel, these bases satisfy the positivity condition : when written in the form $\sum h_wT_w$ then the $h_w$ have positive coefficients.  

The structure of the paper is as follows : In Section 2 we define Soergel bimodules and discuss Soergel's conjecture. In section 3 we give a way to calculate the morphism $f_{sr}$ via symmetric algebras. In section 4 we define $E_x$ and discuss its properties. Finally in section 5 we define $f_{sr}^2(n)$, we define $D_x$ and we discuss its properties.

\section{Soergel conjecture}

\subsection{Soergel category $\mathbf{B}$}
Let $(W,\mathcal{S})$ be a not necessarily finite Coxeter system (with $\mathcal{S}$ a finite set) and $\mathcal{T}\subset W$ the set of reflections in $W$, \textit{i.e.} the orbit of $\mathcal{S}$ under conjugation. Let $k$ be an infinite field of characteristic different from $2$ and $V$ a finite dimensional $k$-representation of $W$. For $w\in W$, we denote by $V^w\subset V$ the set of $w-$fixed points. The following definition might be found in \cite{S3} :

  \begin{definition}
 By a \textit{reflection faithful representation}  of $(W,\mathcal{S})$ we mean a faithful, finite dimensional representation $V$ of $W$ such that for each $w\in W$, the subspace $V^w$ is an hyperplane of $V$ if and only if $w\in\mathcal{T}.$
\end{definition}
 From now on, we consider $V$ a reflection faithful representation of $W$. If $k=\mathbb{R}$, by the results of  \cite{lib}, all the results in this paper will stay true if  we consider $V$ to be the geometric representation of $W$ (still if this representation in not always reflection faithful). So if the reader is not interested in positive characteristic or if he is only interested in the positivity of the coefficients of  Kazhdan-Lusztig polynomials he can suppose from now on that  $k=\mathbb{R}$ and that $V$ is the geometric representation of $W$.

Let $R=R(V)$\label{d1} be the algebra of regular functions on $V$. The algebra $R$ has the following grading : $R=\bigoplus_{i\in
\mathbb{Z}}R_i$ with  $R_2 = V^*$ and $R_i=0$ if $i$ is odd. The 
action of $W$ on $V$ induces an action on $R$. For $s\in\mathcal{S}$ consider the $(R,R)-$bimodule $\theta_s=R\otimes_{R^s} R$, where $R^s$ is the subspace of $R$ fixed by $s$.
 
\begin{definition}
For every graded object $M=\bigoplus_i M_i,$ and every integer $n$, we define the shifted object $M(n)$ by $(M(n))_i=M_{i+n}.$
\end{definition}

\begin{definition}
 Let $\mathcal{R}$ denote the category of all $\mathbb{Z}$-graded $R$-bimodules, which are finitely generated from the right as well as from the left, and where the action of $k$ from the right and from the left is the same.
\end{definition}

Now we can define Soergel bimodule category :

  \begin{definition}
 Soergel's category $\mathbf{B}(W,V)=\mathbf{B}$ is the full subcategory of $\mathcal{R}$ with objects the finite direct sums of direct summands of bimodules of the type ${\theta}_{s_1}\otimes_{{R}}{\theta}_{s_2}\otimes_{{R}}\cdots\otimes_{{R}} {\theta}_{s_n}(d)$ for $(s_1,\ldots, s_n)\in \mathcal{S}^n,$ and $d\in\mathbb{Z}.$ 

By convention, we will denote by ${\theta}_{s_1}{\theta}_{s_2}\cdots  {\theta}_{s_n}$ the $({R},{R})-$bimodule $${\theta}_{s_1}\otimes_{{R}}{\theta}_{s_2}\otimes_{{R}}\cdots\otimes_{{R}} {\theta}_{s_n}\cong {R}\otimes_{{R}^{s_1}}{R}\otimes_{{R}^{s_2}}\cdots \otimes_{{R}^{s_n}}{R}.$$
\end{definition}

\subsection{Soergel categorification and conjecture}

Before we can state Soergel's categorification of the Hecke algebra we will recall what the Hecke algebra and the split Grothendieck group are.

\begin{definition}
Let  $(W, \mathcal{S})$ be a Coxeter system. We define the Hecke algebra $\mathcal{H} = \mathcal{H}(W, \mathcal{S})$
as the $\mathbb{Z} [v,v^{-1}]$-algebra with generators  $\{
 T_{s}
\}_{s\in \mathcal{S}}$, and relations  $$T^{2}_{s}=
v^{-2}+(v^{-2}-1)T_{s} \mathrm{ \ \ for\ all\ \ } s\in \mathcal{S}\ \   \mathrm{and} $$
$$\underbrace{T_{s}T_{r}T_{s}...}_{m(s,r)\, \mathrm{terms}
}=\underbrace{T_{r}T_{s}T_{r}...}_{m(s,r)\, \mathrm{terms} } \ \mathrm{if}\
s,r \in \mathcal{S}\ \mathrm{ and\ }sr \mathrm{\ is\ of\ order\ } m(s,r).$$

If
$x=s_{1}s_{2}\cdots s_{n}$ is a reduced expression of $x$, we define $T_x=T_{s_1}T_{s_2}\cdots T_{s_n}$ ($T_x$ does not depend on the choice of the reduced expression). We will note $q=v^{-2}.$

\end{definition}

\begin{definition}
For every essentially small abelian category $\mathcal{A}$, we define the split Grothendieck 
group $\langle\mathcal{A}\rangle$ : it is the free abelian group generated by the objects of  $\mathcal{A}$ modulo the relations $M=M'+M''$ whenever we have $M\cong M'\oplus M''$. Given an object $A\in \mathcal{A},$ let  $\langle A \rangle$ denote its class in $ \langle \mathcal{A}
\rangle$.
\end{definition}

The following theorem can be found in \cite{S3} :

\begin{theorem}
Let $(W,\mathcal{S})$ be a Coxeter system and $\mathcal{H}$ its Hecke algebra. There exists a unique ring isomorphism
$\varepsilon : \mathcal{H}\rightarrow \langle \mathbf{B}\rangle $
such that $\varepsilon(v)=\langle R(1)\rangle$ and $\varepsilon(T_s+1)=\langle \theta_s \rangle$ for all $ s\in \mathcal{S}. $
\end{theorem}

The following theorem can be found in \cite{KL1}

\begin{theorem}\label{cocoroco}
Let us define in the Hecke algebra the elements $\tilde{T}_x = v^{l(x)} T_x$. There exists a unique involution $d : \mathcal{H} \rightarrow \mathcal{H}$ with $d(v) = v^{-1}$ and $ d(T_x ) = (T_{x^{-1}})^{-1}$. For $x \in W$ there exists a unique $C'_x \in \mathcal{H}$ with $d(C'_x) = C'_x$
 and $$C'_x \in \tilde{T}_x+ \sum_y v\mathbb{Z}[v]\tilde{T}_y$$
 These elements form the so-called Kazhdan-Lusztig basis of the Hecke algebra.
\end{theorem}

Now we can state Soergel's conjecture.
\begin{conjecture}[Soergel]\label{cs}
For every  $x\in W$, there exists an indecomposable $\mathbb{Z}-$graded $R$-bimodule  $B_x\in \mathcal{R}$ such that $\varepsilon(C'_x)=
\langle B_x\rangle$.
\end{conjecture}

\begin{remark}
In \cite{S3} Soergel proves that this conjecture  implies the positivity of all coefficients of all Kazhdan-Lusztig polynomials. By the work of Soergel \cite{S2} and Fiebig \cite{Fi} we know that for $W$ a finite Weyl group and char($k$) at least the Coxeter number, this conjecture is equivalent to  Lusztig's conjecture concerning characters of irreducible representations of algebraic groups over $k$.
\end{remark}

\section{A categorification of the braid relation}\label{olb}
\subsection{Definition of $f_{sr}$}
In this section we will consider $s,r\in \mathcal{S}$, $s\neq r$, with $m(s,r)=m \neq \infty$.  
We put 

\begin{notation}\label{nadal}
$t=\begin{cases}s\ \mathrm{if}\ m\ \mathrm{is\ odd}\\ r\ \mathrm{if}\ m\ \mathrm{is\ even} \end{cases}$ and \hspace{0.3cm} $u=\begin{cases}r\ \mathrm{if}\ m\ \mathrm{is\ odd}\\ s\ \mathrm{if}\ m\ \mathrm{is\ even} \end{cases}$
\end{notation}

and
 $$X_{sr}:=\underbrace{\theta_s\theta_r\theta_s\cdots \theta_t}_{m(s,r)\ \mathrm{terms}} \,\ \  \ \ \ \  \  \ \, X_{rs}:=\underbrace{\theta_r\theta_s\theta_r\cdots \theta_u}_{m(s,r)\ \mathrm{terms}}.$$\label{lab0} 
\begin{definition}
 Let $s\neq r\in \mathcal{S}$ with $m(s,r)\neq \infty.$
We will call $DZ_{sr}$ the space of  degree zero morphisms of $(R,R)$-bimodules from $X_{sr}$ to $X_{rs}$.
\end{definition}

The following proposition can be found in \cite[prop. 4.3]{L} :

  \begin{proposition}\label{literato} Let $s\neq r\in \mathcal{S}$ with $m(s,r)\neq \infty.$
The space $DZ_{sr}$ is one-dimensional.
\end{proposition}

\begin{definition} Recall that $\mathcal{T} \subseteq W$ is the subset of reflections of $W$. For each $t\in\mathcal{T}$, let $Y_t$ be the subset of $ V^{\ast}$ of linear forms with kernel equal to the hyperplane fixed by $t$. 
\end{definition}
It is clear that if  $y, y'\in Y_t$, then there exists $0\neq\lambda \in k$ such that $y=\lambda y'.$ Let $s, r \in \mathcal{S}.$ If we chose one element $x_s\in Y_s$ and one element $x_r\in Y_r$, by \cite[lemma 4.7]{L} there exists a unique  element $f_{sr}\in DZ_{sr} $  with 
$$f_{sr}( 1\otimes x_s\otimes x_r\otimes x_s\otimes \cdots \otimes x_t)\in 1\otimes x_r\otimes x_s\otimes x_r\otimes \cdots \otimes x_u +R_+X_{rs}. $$
Where $R_+$ is the ideal of $R$ generated by the homogeneous elements of non zero degree.  We insist in the fact that $f_{sr}$ depends on the choice of $x_s$ and of $x_r$.

\subsection{A formula for $f_{sr}$}

In this section we find a formula for $f_{sr}$, using the fact that  $R$ is a symmetric algebra over $R^{\langle s,r\rangle}$. We fix a dihedral Coxeter system  $(W,\mathcal{S})$ with $\mathcal{S}=\{s,r\}$ and $m=m(s,r)$. We will suppose that the order of $W$ is invertible in the field $k$.

In this section we will fix, for every $t\in \mathcal{T}$, an element $x_t\in Y_t$.
For all $s\in\mathcal{S}$ we define $\partial_s: R\rightarrow R$ as follows : $\partial_s(a)=(a-s\cdot a)/2x_s$. 
We recall the following theorem  (see \cite[Thm. 2]{De}) :

\begin{theorem}\label{erudito}
\begin{enumerate}
 \item If $w\in W$ and $(s_1,\ldots, s_n)$ is a reduced expression of $w$ then the element $\partial_w=\partial_{s_1}\cdots \partial_{s_n}$\label{H0} depends only on $w$ ; it does not depend on the choice of the reduced expression. 
\item If  $d=\prod_{t\in\mathcal{T}}x_t,$\label{H1}  then the set $\{\partial_w(d)\}_{w\in W} $ is a basis of $R$ as $R^W$-module.
\end{enumerate}
\end{theorem}
So we can define the graded $R^W$\label{H2}-module morphism  
\begin{displaymath}
\begin{array}{lll}
\hat{t} : \ \ \ \ \ \ \ \ \ R &\rightarrow &  R^W(-2m) \\
\sum_{w\in W}\lambda_w\partial_w(d)&\mapsto & \lambda_1
\end{array}
\end{displaymath} The following lemma is classical :
  \begin{lemma}
$R$ is a symmetric algebra over $R^W$ and $\hat{t}$ is the symmetrising form. 
\end{lemma}
\begin{proof}
 By definition  $\hat{t}$ is a linear form and as $R$ is commutative, it is trivial to see that $\hat{t}$ is a trace. We only need to prove that the map  $\psi : R(2m)\rightarrow \mathrm{Hom}_{R^W}(R, R^W)$ that sends $a$ to $(b\rightarrow \hat{t}(ab))$ is a graded isomorphism of  $R$-modules. The fact that $\psi$ is a morphism of $R$-modules is a consequence of the fact that  $R$ is commutative. 

To see that  $\psi$ is injective, we use the easily verifiable fact that $\hat{t}(\partial_x(d)\partial_y(d))\in \delta_{xw_0,y}+R_+$.  Finally, using the following isomorphism (see \cite[ch. IV, cor. 1.11 a.]{Hi})  :
\begin{equation}\label{nico}R \cong \bigoplus_{w\in W} R^{W}(-2l(w))  \textrm{ as\ left\ graded } 
 R^{W}-\mathrm{mod}  \end{equation}
 we conclude easily that there exists an isomorphism of graded $R^W$-modules : $\mathrm{Hom}_{R^W}(R, R^W)\cong R(2m)$, and so  $\psi$ is an isomorphism.
  \end{proof}

Let us consider the dual basis $\{\partial_w(d)^{\ast}\}_{w\in W}$\label{H3} with respect to the linear form $\hat{t}$ :
$$\hat{t}(\partial_w(d)\partial_{w'}(d)^{\ast})=\delta_{w,w'}. $$
The objects of this dual basis are homogeneous.   By \cite[lemma 3.2]{Br}, we have that $\sum_{w\in W}\partial_w(d)\otimes \partial_{w}(d)^{\ast}$ is the Casimir element and it is clear that  deg($\partial_w(d)+\partial_{w'}(d)^{\ast}$)=$2m,$ so, by \cite[proposition 3.3]{Br}, the map
\begin{displaymath}\label{H4}
\begin{array}{lll}
 \xi :  R &\rightarrow &  R\otimes_{R^W}R (-2m) \\
   \hspace{0.5cm}1 &\mapsto & \sum_{w\in W}\partial_w(d)\otimes \partial_{w}(d)^{\ast}
\end{array}
\end{displaymath}is a non-zero morphism of graded
 $(R,R)$-bimodules.

 Let us consider the decomposition $R=R^s \oplus x_s R^s$. Let $p_1, p_2\in R^s$ and $p,q,r\in R$. We define the morphisms of graded $R^s$-modules :
\begin{displaymath}\label{la1}
\begin{array}{lll}
 P_s : R \rightarrow R,&&  p_1+x_sp_2 \mapsto p_1  \\ \smallskip
 I_s : R \rightarrow R,&&  p_1+x_sp_2 \mapsto x_sp_2\\ \smallskip
 \partial_s : R(2) \rightarrow R,&& p_1+x_sp_2 \mapsto p_2\\ \smallskip
\end{array}\end{displaymath} and the morphisms of graded $(R,R)-$bimodules  :
\begin{displaymath}
\begin{array}{lll}\label{la222}
m_s : \theta_s\rightarrow R, && R\otimes_{R^s}R\ni p\otimes q \mapsto pq\\ \smallskip
 j_s : \theta_s\theta_s (2)\rightarrow \theta_s, && R\otimes_{R^s}R\otimes_{R^s}R \ni p\otimes q\otimes r \mapsto p\partial_s(q)\otimes r \in R\otimes_{R^s}R.
\end{array}\end{displaymath}
We can now define two morphisms of graded  $(R,R)$-bimodules (see notation \ref{nadal}) : 
\begin{displaymath}\label{H45}
\begin{array}{llcl}
 \iota :   R\otimes_{R^W}R (-2m) &\rightarrow & R\otimes_{R^t}R\otimes \cdots \otimes_{R^s}R \otimes_{R^W}R (-2m) &\simeq X_{tu} \otimes_{R^W}R (-2m)\\
  \hspace{1.5cm} a\otimes b &\mapsto & a\otimes 1\otimes 1\otimes \cdots \otimes 1\otimes b &
\end{array}\end{displaymath}
and $\vartheta :  X_{sr}\otimes_R X_{tu} \rightarrow   R(-2m)$ defined by
 $$ \vartheta = (m_s\circ j_s)\circ(\mathrm{id}^{1}\otimes (m_r\circ j_r)\otimes \mathrm{id}^{1})\circ(\mathrm{id}^{2}\otimes (m_s\circ j_s)\otimes \mathrm{id}^{2})\circ \cdots \circ (\mathrm{id}^{m-1}\otimes (m_t\circ j_t)\otimes \mathrm{id}^{m-1}) $$

%
%


We introduce the morphism 
$\Phi \in \mathrm{Hom}(X_{sr},R\otimes_{R^W}R )$, defined by $$\Phi =  (\vartheta\otimes \mathrm{id}_{R \otimes_{R^W}R})\circ (\mathrm{id}_{X_{sr}}\otimes (\iota\circ \xi)),$$ where we identify the domain $X_{sr}$ with $X_{sr}\otimes_R R.$ 

Finally we can define the following graded morphism 
\begin{displaymath}
\begin{array}{lll}
 \Psi : R\otimes_{R^W}R  &\rightarrow & \hspace{1.5cm}X_{rs} \\
  \hspace{0.9cm} a\otimes b &\mapsto & a\otimes 1\otimes 1\otimes \cdots \otimes 1\otimes b
\end{array}
\end{displaymath}

  \begin{proposition}\label{calculo}
The morphism $\Psi \circ \Phi \in \mathrm{Hom}(X_{sr},X_{rs})$ is a non-zero scalar multiple of  $f_{sr}.$ If we identify  $X_{sr}=R\otimes_{R^s}R\otimes _{R^r}R\otimes_{R^s}R \cdots \otimes_{R^t}R$ and $X_{rs}=R\otimes_{R^r}R\otimes _{R^s}R\otimes_{R^r}R \cdots \otimes_{R^u}R$ we  explicitly have :
$$ \Psi \circ \Phi(p_0\otimes p_1\otimes\cdots \otimes p_n)= \sum_{w\in W} p_0\partial_s(p_1\partial_r(p_2\partial_s(p_3\cdots \partial_t(p_n\partial_w(d)))\cdots )\otimes 1\otimes 1\otimes \cdots \otimes 1\otimes \partial_w(d)^*
$$


\end{proposition}
\begin{proof}
It is clear that $\Psi \circ \Phi $ is a degree zero morphism, then by definition of $f_{sr}$ we only need to prove that 
$\Psi \circ \Phi \neq 0. $ But for this we only need to note that  $\Psi \circ \Phi(1\otimes 1\otimes \cdots \otimes 1)= \partial_{w_0}(d)\otimes 1\otimes \cdots \otimes 1,$ where $w_0$ is the longest element of $W$. As $\partial_{w_0}(d)$ is part of the basis in theorem \ref{erudito}, it is non zero.  \end{proof}

\begin{remark}
If $\Psi \circ \Phi$ is a scalar multiple of $ f_{sr}$, we can chose some $t_0\in \mathcal{T}$ with $t_0\notin \mathcal{S}$ and change the definition of $x_{t_0}$ (by a scalar multiple) so as to change by a scalar multiple the definition of $d=\prod_{t\in\mathcal{T}}x_t$ and thus to have exactly $\Psi \circ \Phi =f_{sr}.$
\end{remark}

\subsection{An important property of $f_{sr}$}\label{ole}
In this section we prove a property of 
$f_{sr}$ that will be  useful in  section \ref{dromo}.
We start with a trivial corollary of proposition  \ref{calculo} :
  \begin{corollary}\label{fsr}
 $f_{sr}(1\otimes_{R^s}\theta_r\theta_s\cdots \theta_t)\subseteq R^s\otimes_{R^r}1\otimes_{R^s}1\otimes \cdots 1\otimes_{R^u}R$
\end{corollary}

We need some definitions in order to state the next proposition.
  \begin{definition}\label{iriarte}
We define the morphism $$^i\hspace{-0.09cm}f=\mathrm{id}^i\otimes f_{sr}\otimes \mathrm{id}^{n-i-m(s,r)}\in\mathrm{Hom}(\theta_{s_1}\cdots \theta_{s_n},\theta_{t_1}\cdots \theta_{t_n}),$$\label{lab6} whenever this makes sense.
A morphism $g$  between $\theta_{s_1}\cdots \theta_{s_n}$ and $\theta_{t_1}\cdots \theta_{t_n}$ is of $f$-type \label{lab7} if there exists a sequence $\bar{i}=(i_1,\ldots, i_k)$ such that $$g=^{i_k}\hspace{-0.17cm}f\circ \cdots\circ ^{i_2}\hspace{-0.17cm}f\circ ^{i_1}\hspace{-0.17cm}f.$$
\end{definition}
\vspace{0.2cm} 
\begin{definition}\label{K165}
Let $g, g_0, g_1,\ldots, g_m$ be  morphisms in Soergel's category $\mathbf{B}$. We say that the tuple $(g_m,  \ldots, g_1, g_0)$ is an expression of $g$ if $g=g_m\circ \cdots \circ g_1\circ g_0$. 
 \end{definition}
\vspace{0.2cm} 
 \begin{definition}\label{lab8}
 Let $\bar{s}=(s_1,\ldots,s_n)\in \mathcal{S}^n$. We define $\theta_{\bar{s}}=\theta_{s_1}\cdots\theta_{s_n}.$
\end{definition}

  \begin{definition}
Let $(i_k,\ldots,i_2,i_1)\in\mathbb{N}^k$. We define $$\mathrm{Null}(i_k,\ldots,i_1)=\{j\ \vert\ 1\leq j\leq k\ \mathrm{and}\ i_j=0 \}$$
\end{definition}
  \begin{proposition}\label{dimmensdale}
 Let us fix an integer $p\geq 2$. Let us consider $\bar{t}=(t_1,\ldots, t_{p-1})$ and $\bar{a}=(a_1,\ldots, a_{p-1})$, two reduced expressions of  $x\in W$. Let  $g\in \mathrm{Hom}(\theta_s\theta_{\bar{t}},\theta_s\theta_{\bar{a}} )$ be an $f$-type morphism. We have the following inclusion $g(1\otimes_{R^s}\theta_{\bar{t}})\subseteq 1\otimes_{R^s}\theta_{\bar{a}}$. 
\end{proposition}
\begin{proof} Let    $g=^{i_k}\hspace{-0.17cm}f\circ \cdots\circ ^{i_2}\hspace{-0.05cm}f\circ ^{i_1}\hspace{-0.08cm}f$. If we consider the domain and the co-domain of $g$, we can conclude that the set $\mathrm{Null}(i_k,\ldots,i_1)$ has an even cardinality. 

Let $\mathrm{Null}(i_k,\ldots,i_1)=\{y_1,y_2,\ldots, y_{2\alpha}\}$ with $y_1<y_2<\ldots <y_{2\alpha}$ and $\alpha \geq 0$. We will prove the proposition by double induction in  $p$ and in  $\alpha.$ We will use the notation  $T(p_0,\alpha_0)$ if the proposition is true for $p=p_0$ and $\alpha=\alpha_0.$ So we have to prove 
\begin{enumerate}
 \item $T(i,0)$ for all $i\geq 2.$ 
\item  $(T(i,\mathbb{N}) \mathrm{\ for \ all\ }2\leq i<p)\Rightarrow (\mathrm{for \ all\ }\alpha\geq 0,\ T(p,\alpha)\Rightarrow T(p,\alpha+1)  ), $
\end{enumerate}
where $\mathbb{N}$ is the set of non-negative integers.

The first assertion is trivial. We will prove the second assertion. We suppose  $T(i,\mathbb{N}) \mathrm{\ for \ all\ }2\leq i<p$ and $T(p,\alpha)$, and we will prove  $T(p,\alpha+1)$. We define $w=\begin{cases}t_1\ \mathrm{if}\ m(s,t_1)\ \mathrm{is \ odd}\\ s\ \mathrm{if}\ m(s,t_1)\ \mathrm{is\ even}.\end{cases}$

We define the three morphisms  $H,F,G$ by the formula  $$H=^{i_{y_2}}\hspace{-0.17cm}f\circ \cdots\circ ^{i_2}\hspace{-0.17cm}f\circ ^{i_1}\hspace{-0.17cm}f=^0\hspace{-0.17cm}f\circ F\circ ^0\hspace{-0.17cm}f\circ G.$$
We define $\bar{t'},\ \bar{t''},\ \bar{a'},\ \bar{a''}$    in the following way :  

\begin{itemize}
\item $G\in \mathrm{Hom}(\theta_{s}\theta_{\bar{t}},\theta_{s}\theta_{\bar{t'}})$ 
\item  $F\in \mathrm{Hom}(\underbrace{\theta_{t_1}\theta_{s}\theta_{t_1}\cdots \theta_{w}}_{m(s,r)\ \mathrm{terms}}\theta_{\bar{t''}}, \underbrace{\theta_{t_1}\theta_{s}\theta_{t_1}\cdots \theta_{w}}_{m(s,r)\ \mathrm{terms}}\theta_{\bar{a'}})$ 
\item  $H\in \mathrm{Hom}(\theta_{s}\theta_{\bar{t}},\theta_{s}\theta_{\bar{a''}})$ 
\end{itemize}
It is clear that  $G(1\otimes_{R^s}\theta_{\bar{t}})\subseteq 1\otimes_{R^s}\theta_{\bar{t'}}$. By the corollary \ref{fsr}, $$^0\hspace{-0.05cm}f\circ G(1\otimes_{R^s}\theta_{\bar{t}})\subseteq R^s\otimes_{R^{t_1}}1\otimes_{R^s}1\otimes_{R^{t_1}} \cdots 1  \otimes_{R^w}\theta_{\bar{t''}}.$$

By induction hypothesis, $T(i,\mathbb{N})\ \mathrm{for \ all\ }2\leq i<p$, so in particular $T(p-m(s,t_1),\mathbb{N})$, and this implies $$ F\circ ^0\hspace{-0.17cm}f\circ G(1\otimes_{R^s}\theta_{\bar{t}})\subseteq R^s\otimes_{R^{t_1}}1\otimes_{R^s}1\otimes_{R^{t_1}}\cdots  1\otimes_{R^w}\theta_{\bar{a'}},$$ and finally $H(1\otimes_{R^s}\theta_{\bar{t}})\subseteq R^s\otimes_{R^s}\theta_{\bar{a''}} \subseteq 1\otimes_{R^s}\theta_{\bar{a''}}$. Again by induction hypothesis we know $T(p,\alpha)$, so if $H'=^{i_k}\hspace{-0.17cm}f\circ \cdots\circ ^{i_{y_2+2}}\hspace{-0.17cm}f\circ ^{i_{y_2+1}}\hspace{-0.17cm}f$, we have $H'(1\otimes_{R^s}\theta_{\bar{a''}})\subseteq 1\otimes_{R^s}\theta_{\bar{a}},$ so $$g(1\otimes_{R^s}\theta_{\bar{t}})=H'\circ H(1\otimes_{R^s}\theta_{\bar{t}})\subseteq 1\otimes_{R^s}\theta_{\bar{a}}.$$ \end{proof}

\section{The bimodule $E_w$}\label{dromo}

From now on we will assume that our Coxeter group $W$ is extra-large, \textit{i.e.} $m(s,r)>3$ for all $s,r\in\mathcal{S}$.
Until the end of this section we will fix  a reduced expression $\bar{s}=(s_1,\cdots, s_n)\in \mathcal{S}^n$ of an element  $x\in W$. In this section all morphisms will be assumed to be of $f$-type (see definition \ref{iriarte}). 
\subsection{Towards $E_w$}
We start with some definitions and notations.

\begin{notation}\label{orden} We say that an integer interval is a set of positive integers of the form $I=\{a,a+1,a+2,\ldots, b\}$. It might be the empty set. We will use the standard notation $I=\intl a,b\intr$ (in this notation we assume $a\leq b$). If $a>b$ we define $\intl a,b\intr=\emptyset.$

\end{notation}

\begin{definition}\label{todo} Let  $\bar{r}=(r_1,\cdots, r_n)\in \mathcal{S}^n$ be a reduced expression of an element $x$ of $W$. We define the set 
$$ T(\bar{r}):= \left\{\begin{array}{lll}
I\ \mathrm{integer\ interval\ }&\vert &  \ \mathrm{if}\ i,i+2\in I, \mathrm{then}\ r_i=r_{i+2}                  \\
             \end{array}\right\}$$
\end{definition}
\begin{definition}
We define the set of integer intervals $A(x)$, as the subset of $ \bigcup_{\bar{r}\in \mathcal{R}(x)}T(\bar{r})$ consisting of all maximal intervals with respect to the partial order given by inclusion, where the symbol $\mathcal{R}(x)$ stands for the set of all reduced expressions of $x$. 
\end{definition}

\begin{definition}\label{cores}
 If $A(x)=\bigcup_{j\in J}\intl a_j,b_j \intr$ then we define :
\begin{itemize}
\item $\mathrm{Gcores}(x)=\bigcup_{j\in J}\intl a_j+1,  b_j-1\intr$, the set of generalized cores. 
\item $\mathrm{ELGcores}(x)=\{\intl a,  b\intr\in \mathrm{Gcores}(x)\ \vert\ b-a\geq 1 \}$  the set of extra-large generalized cores. 
\item $\mathrm{cores}(x)= \{\intl a, b\intr\in \mathrm{Gcores}(x) \ \vert\ m(s_a,s_{a+1})=b-a+3\},$ the set of cores, where $(s_1,\ldots, s_n)$ is some reduced expression of $x$. It is easy to see that this set does not depends on the choice of the expression $(s_1,\ldots, s_n)$.
\end{itemize}
\end{definition}

\subsection{}

We define a total order in the set $\mathrm{cores}(x)$ in the following way : 
If $C=\intl a,b\intr, C'=\intl a',b'\intr \in \mathrm{cores}(x)$, we say that $C<C'$ if and only if $b<a'$. We remark that it is a total order because $W$ is extra-large.
We define the distance from $C$ to $C'$ by $$ \mathrm{dist}(C,C')=\mathrm{min}\{\vert a'-b\vert,\vert a-b'\vert\}.$$\label{K2} 

  \begin{definition}\label{K4}
 If $C=\intl a, b\intr\in \mathrm{cores}(x)$ we define first($C$)=$a$\label{K3} and last($C$)=$b$. The core $C$ is called right (resp. left) core if $s_{b-1}\neq s_{b+1}$ and $s_{a-1} = s_{a+1}$ (resp. $s_{b-1} = s_{b+1}$ and $s_{a-1}\neq s_{a+1}$). It is called empty core if  $s_{b-1}\neq s_{b+1}$ and $s_{a-1} \neq s_{a+1}$  and it is called filled core if it is none of the previous ones ($i.e.\ s_{b-1}= s_{b+1}, s_{a-1}= s_{a+1} $).
\end{definition}

\begin{remark}\label{filled}
 If $C<C'$, dist($C,C')=1$ and $C$ is a right or empty core, then $C'$ must be a right or filled core.
\end{remark}

 The following lemma is easy : 

  \begin{lemma}\label{herodoto}
 Let $ \mathrm{cores}(x)=\{C_1,C_2,\ldots, C_k\}$ with $C_1<C_2<\ldots <C_k,$ and let $g\in\mathrm{End}(\theta_{\bar{s}})$. Suppose that $ \mathrm{dist}(C_i,C_{i+1})\geq 2$ for some $1\leq i\leq k$. If last($C_i$)=$d-1$, $\bar{s'}=(s_1,\ldots, s_d), \, \bar{s''}=(s_{d+1},\ldots, s_n)$, then there exists $g'\in \mathrm{End}(\theta_{\bar{s'}})$, $g''\in \mathrm{End}(\theta_{\bar{s''}})$ such that $g=g'\otimes g''.$ 
\end{lemma}

The following proposition is very important for the sequel.

  \begin{proposition}\label{aristoteles}
 Let  $C', C''\in \mathrm{cores}(x)$, $C'<C''$,  $ \mathrm{dist}(C',C'')=1$ and  $s_{d-2}=s_d=s_{d+2}$, for $d=$last$(C')+1$. Let us put $\bar{s'}=s_1\cdots s_d, \, \bar{s''}=s_{d}\cdots s_n$. If $g\in \mathrm{End}(\theta_{\bar{s}})$ is of $f$-type, there exists $g'\in \mathrm{End}(\theta_{\bar{s'}})$ and   $g''\in \mathrm{End}(\theta_{\bar{s''}})$ of $f$-type such that
\begin{equation}\label{chuchoca}
 g=\ (g'\otimes \mathrm{id}^{n-d+1})\circ (\mathrm{id}^d\otimes g'')=(\mathrm{id}^d\otimes g'')\circ (g'\otimes \mathrm{id}^{n-d+1})
\end{equation}
\end{proposition}

  \begin{remark}\label{onasis}
 We remark that it is equivalent to say $s_{d-2}=s_d=s_{d+2}$ or to say that $C'$ is a left or a filled core and $C''$ is a right or a filled core.
\end{remark}
\begin{proof}
  Let $b+1=$first($C'$) and $g=^{i_k}\hspace{-0.17cm}f\circ \cdots\circ ^{i_2}\hspace{-0.17cm}f\circ ^{i_1}\hspace{-0.17cm}f$. Consider the set $X(C',C'')=\{1\leq p\leq k\ \vert\ i_p=b-1\mathrm{\ or\ }i_p=d-1\}$ and let $X(C',C'')=\{x_1,x_2,\ldots ,x_{2t}\}$ with $x_1<x_2<\ldots <x_{2t}$. As $s_{d-2}=s_d=s_{d+2}$, we can easily deduce by induction on $l$ that $i_{x_{2l}}=i_{x_{2l-1}}$ for $1\leq l\leq t,$ and that is the reason why $X(C',C'')$ has an even number of elements. We define $x_0=0$ and $x_{2t+2}=k+1$. For $0\leq c\leq t$ we define  the sets $$Z_c^<=\{p\ \vert  x_{2c}<p\leq x_{2c+2}\  \}, $$
$$Z_c^<=Z_c\cap \{ p\ \vert\ i_p<d-1 \} $$ and $$Z_c^{\geq}=Z_c\cap \{p\ \vert\ i_p\geq d-1 \}. $$
For $1\leq c\leq t$ we define $L_c=^{i_{\alpha}}\hspace{-0.17cm}f\circ  ^{i_{\beta}}\hspace{-0.17cm}f\circ \cdots\circ  ^{i_{\gamma}}\hspace{-0.17cm}f$ where the elements of $Z_c^<$ are $\alpha<\beta<\ldots <\gamma,$ and $R_c=^{i_{\delta}}\hspace{-0.17cm}f\circ  ^{i_{\epsilon}}\hspace{-0.17cm}f\circ \cdots\circ  ^{i_{\omega}}\hspace{-0.17cm}f$ where the elements of $Z_c^{\geq}$ are $\delta<\epsilon<\ldots <\omega.$

It is clear that if $z\in Z_c^<$ and $w\in Z_c^{\geq}$ then $^{i_{z}}\hspace{-0.12cm}f$  commutes with $^{i_{w}}\hspace{-0.12cm}f$, so $$L_c\circ R_c=R_c\circ L_c=^{i_{\rho}}\hspace{-0.17cm}f\circ \cdots ^{i_{\sigma}}\hspace{-0.17cm}f\circ ^{i_{\tau}}\hspace{-0.17cm}f,$$  where the elements of $Z_c$ are $\rho<\sigma<\ldots <\tau.$

We define $M=\theta_{s_1}\theta_{s_2}\cdots \theta_{s_{d-1}}$ and $N=\theta_{s_{d+1}}\cdots \theta_{s_n}$. So we have $\theta_{\bar{s}}=M\otimes_{R^{s_{d}}}N.$ By proposition \ref{dimmensdale} we have that for all $0\leq c\leq t$ if $m\otimes n\in M\otimes_{R^{s_{d}}}N$, then $L_c(m\otimes n)\subseteq  M\otimes n$ and $R_c(m\otimes n)\subseteq m\otimes N$, so for all $0\leq a,c\leq  t $ we have
\begin{equation}\label{hocico}
 R_a\circ L_c =L_c\circ R_a.
\end{equation}
 If we define $g'=L_t\circ L_{t-1}\circ \cdots \circ L_0$ and $g''=R_t\circ R_{t-1}\circ \cdots \circ R_0$ then equation (\ref{hocico}) allows us to finish the proof. \end{proof}

Before we state the next theorem we have to make a definition.
  \begin{definition}\label{cosmopolitan}
 Let $C\in \mathrm{cores}(x)$. We say that a tuple of elements $\bar{\alpha}=(i_k,\ldots, i_2,i_1)$, $k\in\mathbb{N}$ is  $\bar{s}$\label{K5}-compatible if the morphism $g=^{i_k}\hspace{-0.17cm}f\circ \cdots\circ ^{i_2}\hspace{-0.17cm}f\circ ^{i_1}\hspace{-0.17cm}f\in \mathrm{End}(\theta_{\bar{s}})$ is well-defined, and in this case we say that $g$ is the morphism associated to $\bar{\alpha}$.  If $\bar{\alpha}$ is $\bar{s}$-compatible we define $N_{\bar{\alpha}}(C)=\mathrm{card}(\{p\ \vert\ i_p=\mathrm{first}(C)-2\})$\label{K6}, where card stands for cardinality.
\end{definition}

  \begin{theorem}\label{huge}
 Let $\bar{\alpha}$ and $\bar{\beta}$ be two $\bar{s}$-compatible tuples, and $g$, $h$ the morphisms associated to $\bar{\alpha}$ and $\bar{\beta}$ respectively.
 We have that $g=h$ if and only if for all $C\in \mathrm{cores}(x)$ we have $N_{\bar{\alpha}}(C)=0\Leftrightarrow N_{\bar{\beta}}(C)=0$.
\end{theorem}
\begin{proof}
 The ``only if'' part is evident. We will prove the ``if'' part. Let $\mathrm{cores}(x)=\{C_1,\ldots , C_k\}$ with $C_1<C_2<\ldots <C_k$. We will prove the theorem by induction over $k$. 

It is easy to see using the definition of $f_{sr}$ that $f_{sr}\circ f_{r,s}\circ f_{sr}=f_{sr}$ and this allows us to prove the theorem for $k=1.$

We suppose the theorem is true for $k-1$ and we will prove it for $k$. If there exists $1\leq i\leq k-1$ such that $ \mathrm{dist}(C_i,C_{i+1})\geq 2$, then lemma \ref{herodoto} and the induction hypothesis allows us to conclude. So we suppose $ \mathrm{dist}(C_i,C_{i+1})=1$ for all $1\leq i\leq k-1$.

If there exists $1\leq i\leq k$ such that $N_{\bar{\alpha}}(C_i)=0$ we can conclude by induction hypothesis, so we suppose $N_{\bar{\alpha}}(C_i)\neq 0$ for all $1\leq i\leq k$.

Let $(j_i)_{1\leq i\leq \gamma}$\label{K7} be an ascending sequence of numbers such that the set of  filled cores is exactly $\{C_{j_i}\}_{1\leq i\leq \gamma}$. 
By proposition \ref{aristoteles}, remark \ref{onasis} and the fact that $ \mathrm{dist}(C_i,C_{i+1})=1$ we see that if $C_i$ is a filled core, then $C_{i+1}$ can not be a filled core, so $j_{i+1}-j_i\geq 2.$ We can conclude that the set $\mathcal{R}_i=\{j_i+1,j_i+2,\ldots , j_{i+1}-1\}$\label{K8} is non empty for all $1\leq i< \gamma$.

 Let us  suppose that there exist an integer $p$ such that for all $i\in \mathcal{R}_p$ the core $C_i$ is not an empty core.  So  if $i\in \mathcal{R}_p$, then $C_i$ is either left or right core. By remark \ref{filled}, we know that there exist an integer $v$, with $j_{p}+1\leq v\leq j_{p+1}$, such that if $j_p+1< i\leq v $ then $C_i$ is a left core and if $v<i<j_{p+1}$ then $C_i$ is a right core.  We are in the case treated in proposition \ref{aristoteles} with $C'=C_v$ and $C''=C_{v+1}$ (see remark \ref{onasis}), so we can conclude by induction hypothesis.

So from now on we suppose that for all integers $1\leq p<\gamma$ there exists some $i_p\in \mathcal{R}_p$\label{K9} such that the core $C_{i_p}$ is an empty core. By remark \ref{filled}, if $j_p<i<i_p$ then $C_i$ is a left core and if $i_p<i<j_{p+1}$, then $C_i$ is a right core. Then $i_p$ is well-defined for $1\leq p< \gamma$.


 Let $\bar{\alpha}=(\alpha_w,\ldots,\alpha_1)$, so  $g=^{\alpha_w}\hspace{-0.17cm}f\circ \cdots\circ ^{\alpha_2}\hspace{-0.17cm}f\circ ^{\alpha_1}\hspace{-0.17cm}f.$ If we consider the set of all  $\bar{s}$-compatible tuples satisfying that $g$ is their associated morphism (see definition \ref{cosmopolitan}), with the partial order $<\hspace{-0.45cm}^{\bullet}\ $\label{K10} given by the length of the tuple, we can assume without loss of generality that $\bar{\alpha}$ and $\bar{\beta}$ are minimal for $<\hspace{-0.45cm}^{\bullet}\ \ .$ 

  \begin{lemma}\label{6.2} For all $1\leq p< \gamma$ we have $N_{\bar{\alpha}}(C_{j_p})=2$ (recall that $C_{j_p}$ is a filled core). 
\end{lemma}
\begin{proof} Let us suppose that this is false, so there exists some $N_{\bar{\alpha}}(C_{j_p})\geq 4.$
By definition this means that the set $Z=\{ 1\leq i\leq w\ \vert\ \alpha_i =\mathrm{first}(C_{j_p})-2 \}$ has more than three elements. We note $a=\mathrm{first}(C_{j_p})-2 .$ Let $z_1<z_2<z_3$ be the first three elements of $Z$. As by definition $C_{j_{p}}$ is a filled core, we have that $C_{j_{p}-1}$ is either a right or an empty core and $C_{j_{p}+1}$ is either a left or an empty core, so we can conclude that  $^{\alpha_{z_3}}\hspace{-0.12cm}f$ commutes with $^{\alpha_{z_3-1}}\hspace{-0.12cm}f\circ \cdots\circ ^{\alpha_{z_2+2}}\hspace{-0.17cm}f\circ ^{\alpha_{z_2+1}}\hspace{-0.17cm}f$, so by commuting this two morphisms we get : $$g=^{\alpha_w}\hspace{-0.17cm}f\circ \cdots\circ \widehat{^{\alpha_{z_3}}\hspace{-0.12cm}f}  \circ \cdots \circ^{\alpha_{z_2+2}}\hspace{-0.17cm}f \circ^{\alpha_{z_2+1}}\hspace{-0.17cm}f\circ (^a\hspace{-0.04cm}f\circ ^a\hspace{-0.17cm}f)\circ^{\alpha_{z_2-1}}\hspace{-0.17cm}f\circ \cdots \circ ^{\alpha_2}\hspace{-0.17cm}f\circ ^{\alpha_1}\hspace{ -0.17cm}f$$
where $\widehat{^{\alpha_{z_3}}\hspace{-0.12cm}f}$ means that we skip this term.  Because of proposition \ref{dimmensdale}, $(^a\hspace{-0.04cm}f\circ ^a\hspace{-0.17cm}f)$ commutes with $^{\alpha_{z_2-1}}\hspace{-0.08cm}f\circ \cdots \circ^{\alpha_{z_1+2}}\hspace{-0.17cm}f \circ^{\alpha_{z_1+1}}\hspace{-0.17cm}f$, so again by commuting this two morphisms we get
$$g= ^{\alpha_w}\hspace{-0.17cm}f\circ \cdots\circ \widehat{^{\alpha_{z_3}}\hspace{-0.12cm}f} \circ \cdots\circ \widehat{^{\alpha_{z_2}}\hspace{-0.12cm}f} \circ \cdots \circ^{\alpha_{z_1+1}}\hspace{-0.17cm}f \circ (^a\hspace{-0.04cm}f\circ ^a\hspace{-0.04cm}f\circ ^a\hspace{-0.17cm}f)\circ^{\alpha_{z_1-1}}\hspace{-0.17cm}f\circ \cdots \circ ^{\alpha_2}\hspace{-0.17cm}f\circ ^{\alpha_1}\hspace{-0.17cm}f$$

but the fact that $^a\hspace{-0.1cm}f\circ ^a\hspace{-0.1cm}f\circ^a\hspace{-0.1cm}f=^a\hspace{-0.1cm}f$ contradicts the minimality of $\bar{\alpha}$ in the $<\hspace{-0.45cm}^{\bullet}\ $ order. So this proves lemma \ref{6.2}. \end{proof}

Let us recall some of the notation we have introduced throughout this proof. Recall that $\bar{s}=(s_1,\ldots, s_n)$, $g=^{\alpha_w}\hspace{-0.17cm}f\circ \cdots\circ ^{\alpha_2}\hspace{-0.17cm}f\circ ^{\alpha_1}\hspace{-0.17cm}f,$ $\mathrm{cores}(x)=\{C_1,\ldots , C_k\}$,  the set of  filled cores is  $\{C_{j_i}\}_{1\leq i\leq \gamma}$ and the set of empty cores is   $\{C_{i_p}\}_{1\leq p< \gamma}$.

  \begin{definition}\label{coimbrero}For $1\leq i\leq k$ we define $f(C_i)= ^y\hspace{-0.17cm}f$, where $y=$first($C_i$)$-2.$
 We define $i_{\gamma}=n$. For $1\leq p\leq \gamma+1$, we consider the set $$V^{\mathrm{left}}_p=\{ 1\leq i\leq w\ \vert\  ^{\alpha_i}\hspace{-0.1cm}f=f(C_q) \ \mathrm{for}\ j_p<q\leq i_p \},$$\label{K11} and put $V^{\mathrm{left}}_p=\{v_1,v_2,\ldots,v_{u(p)}\}$ with $v_1<v_2<\ldots <v_{u(p)}$. We define a sequence $\bar{\alpha}^{\mathrm{left}}(p)=(n_1,n_2,\ldots, n_{u(p)})$\label{K13} in the following way : if $^{\alpha_{v_i}}f=f(C_q)$, we define $n_i=q-j_p.$ 
\end{definition}\label{K15}

\begin{lemma}\label{hermosillo} For $1\leq p\leq \gamma$ we have \begin{equation}\label{diof1}\bar{\alpha}^{\mathrm{left}}(p)=(1,2,\ldots, u(p)/2,u(p)/2,\ldots , 2,1),\end{equation}
 where $u(p)/2=i_p-j_p.$ 
\end{lemma}
\begin{proof}
 Let us prove by induction in $l$ that $(n_1,n_2,\ldots,n_l)=(1,2,\ldots, l)$ for $1\leq l\leq u(p)/2,$ (the proof of $(n_{u(p)/2+1}, n_{u(p)/2+2},\ldots, n_{u(p)/2+l})=(u(p)/2,\ldots, u(p)/2-l+1) $ is similar). As $C_i$ is a left core for $j_p<i<i_p$, it is clear that $n_1=1.$ Let us suppose  $(n_1,n_2,\ldots,n_{l-1})=(1,2,\ldots, l-1)$ for $1\leq l\leq u(p)/2$ : we will prove that $n_{l}=l$. If this is not the case, we must have $n_l=l-1$.  We now have two possibilities for $n_{l+1}$ : it can be $l-1$ or $l-2$. But if $n_{l+1}=l-1$, with commutations relations we will obtain a subexpression of the form $^a\hspace{-0.1cm}f\circ ^a\hspace{-0.1cm}f\circ^a\hspace{-0.1cm}f$, for some integer $a$, and this contradicts the minimality of $\bar{\alpha}$  in the $<\hspace{-0.45cm}^{\bullet}\ $ order. So we conclude that  $n_{l+1}=l-2.$

 As $C_i$ is a left core for $j_p<i<i_p$ we have that $\vert n_i-n_{i+1}\vert \leq 1$ for all $i,$ and the hypothesis we made in this proof that $N_{\bar{\alpha}}(C_i)\neq 0$ for all $1\leq i\leq k$ allows us to conclude that $\{n_i\}_{1\leq i\leq u(p)}=\{1,2,\ldots,u(p)/2\}$. So consider $m$ the minimum of the set $\{ i>l\, \vert\, n_i=l-1 \}. $ Because of the proposition \ref{dimmensdale} we deduce that with commutation relations we can obtain an expression of $g$ with a subexpression of the form  $^{\alpha_{v_{l-1}}}\hspace{-0.1cm}f\circ ^{\alpha_{v_{l-2}}}\hspace{-0.1cm}f\circ ^{\alpha_{v_{m}}}\hspace{-0.1cm}f$, which again contradicts the minimality of $\bar{\alpha}$  in the $<\hspace{-0.45cm}^{\bullet}\ $ order thus proving the lemma.

\end{proof}
We now repeat definition \ref{coimbrero} but changing left by right :
  \begin{definition}We define $i_0=1$.
We now consider, for $0\leq p\leq \gamma $ $$V^{\mathrm{right}}_p=\{ 1\leq i\leq w\ \vert\  ^{\alpha_i}\hspace{-0.1cm}f=f(C_q) \ \mathrm{for}\ i_p\leq q< j_{p+1} \},$$\label{K12} and put $V^{\mathrm{right}}_p=\{d_1,d_2,\ldots,d_{e(p)}\}$ with $d_1<d_2<\ldots <d_{e(p)}$. We define a sequence $\bar{\alpha}^{\mathrm{right}}(p)=(m_1,m_2,\ldots, m_{e(p)})$\label{K14} in the following way : if $^{\alpha_{d_i}}f=f(C_q)$, we define $m_q=j_{p+1}-q.$ 
\end{definition}

By similar arguments as in lemma \ref{hermosillo} we conclude that 
\begin{equation}\label{diof2}\bar{\alpha}^{\mathrm{right}}(p)=(1,2,\ldots, e(p)/2,e(p)/2,\ldots , 2,1),\end{equation}
 where $e(p)/2= j_{p+1}-i_p.$

The two equations (\ref{diof1}) and (\ref{diof2}) allows us to conclude that for all $1\leq p\leq \gamma+1$
\begin{equation}\label{oriundo1}
 \bar{\alpha}^{\mathrm{left}}(p)=\bar{\beta}^{\mathrm{left}}(p)=(1,2,\ldots, u(p)/2,u(p)/2,\ldots , 2,1)
\end{equation}
and for all $0\leq p \leq \gamma $
\begin{equation}\label{oriundo2}
 \bar{\alpha}^{\mathrm{right}}(p)=\bar{\beta}^{\mathrm{right}}(p)=(1,2,\ldots, e(p)/2,e(p)/2,\ldots , 2,1)
\end{equation}

To finish the proof of theorem \ref{huge}, we will only need equations (\ref{oriundo1}) and (\ref{oriundo2}).

Before we can prove  theorem \ref{huge} we need some definitions.
  \begin{definition}\label{K16}
We will say that a morphism $\delta\in$Hom($\theta_{t_1}\cdots \theta_{t_p},\theta_{r_1}\cdots \theta_{r_p}$) is  a $p$-morphism if there exists a sequence of integers $(a_0,\ldots , a_p)$ such that $\delta=\mathrm{id}^{a_0}\otimes f\otimes \mathrm{id}^{a_1}\otimes f\otimes \cdots \mathrm{id}^{a_{p-1}}\otimes f\otimes \mathrm{id}^{a_p}.$
\end{definition}
 
  \begin{definition}\label{K17}
Let $g=(g_m, g_{m-1}, \ldots, g_1)$ and $g'= (g'_d, g'_{d-1},\ldots, g'_1)$ be two expressions of the same morphism (recall \ref{K165}), such that  $g_i$ is a $p_i$-morphism and $g'_i$ is a $p'_i$-morphism. We say that $g<g'$ if and only if the sequence $(p_1,\ldots, p_m)<(p'_1,\ldots, p'_d)$ in the lexicographical order.  We will call this order the ``morphism'' order.
\end{definition}

 \begin{definition}
 Consider two morphisms $g :M_1 \rightarrow N_1$ and $f :M_2\rightarrow N_2.$ The relation $(f\otimes \mathrm{id}_{N_1})\circ (\mathrm{id}_{M_2}\otimes g)= (\mathrm{id}_{N_2}\otimes g)\circ (f\otimes \mathrm{id}_{m_1})$, satisfied in all tensor categories, will be called commutation relation.
\end{definition}

We choose $\bar{g}=(g_m, g_{m-1}, \ldots, g_0)$ an expression of $g$ and $\bar{h}=(h_d, h_{d-1}, \cdots, d_0)$ an expression of $h$, both maximal in the morphism order. We will prove that $d=m$ and that $g_i=h_i$ for all $0\leq i\leq m$. In fact we will prove more, we will  explicitly determine all the $g_i$.  

 We will say that a $p$-morphism $\omega$ acts on the cores $T=\{C_{n_1},C_{n_2},\ldots , C_{n_p}\}$ if  $\omega$ does not act as the identity exactly in that set of cores, and evidently $\omega$ is determined by $T$.

We put $y_p=\mathrm{max}\{i_p-j_p,j_{p+1}-i_p\}$. For all $i\geq 0$ for which this definition is not empty we define the following sets :
$$ T_i^{1,p}=C_{j_p+i< i_p}\cup C_{j_p-i >i_{p-1}}  $$
$$ T_i^{2,p}=\delta_{i,y_p}\cdot C_{i_p}  $$
$$ T_i^{3,p}= \delta_{i,y_p+1}\cdot C_{i_p} $$
$$ T_i^{4,p}=  C_{i_p+y_p-i>j_p}\cup  C_{i_p+i-y_p<j_{p+1}} $$
and $$ T_i=\bigcup_{1\leq p\leq \gamma}T_i^{1,p}\cup T_i^{2,p}\cup T_i^{3,p}\cup T_i^{4,p}.  $$

It is easy to see that if $\omega_i$ acts on $T_i$, then  the expression $\overline{\omega}=(\omega_m, \omega_{m-1}, \ldots, \omega_0)$ is maximal in the morphism order and satisfies equations (\ref{oriundo1}) and (\ref{oriundo2}). So to prove  Theorem \ref{huge} we only need to prove that with commutation relations we can pass from $\bar{g}$ to $\overline{\omega}$. We will prove the following property by induction on $i$ :

Property ($\ast$) : With commutation relations we can pass from $\bar{g}$ to an expression of the form $(^{b_n}\hspace{-0.07cm}f,\ldots ,^{b_1}\hspace{-0.15cm}f,\omega_i, \ldots, \omega_0)$, with $b_1,\ldots, b_n\in \mathbb{Z}$.

 We have that $\omega_0$ is by definition the morphism that acts on the set $\{C_{j_p}\}_{1\leq p\leq \gamma}$, so Property ($\ast$) is clear for $i=0$. 

Let us suppose that Property ($\ast$) is true for $i$, we will prove it for $i+1$. So we have an expression $(^{b_n}\hspace{-0.08cm}f,\ldots ,^{b_1}\hspace{-0.15cm}f,\omega_i, \ldots, \omega_0)$. Let us consider $1\leq p\leq \gamma.$ We have to consider four  cases : 
\begin{enumerate}
 \item  $i< y_p$
\item $i= y_p$
\item $i= y_p+1$
\item $i>y_p$
\end{enumerate}
Let $1\leq a\leq 4.$ By equations (\ref{oriundo1}) and (\ref{oriundo2}) we see that in case $a$ we can make commutation relations in  the subexpression $(^{b_n}\hspace{-0.08cm}f,\ldots ,^{b_1}\hspace{-0.15cm}f)$, and arrive to an expression of the form $(^{b'_u}\hspace{-0.1cm}f,\ldots ,^{b'_1}\hspace{-0.15cm}f, \omega_{i+1}^p)$, where $\omega_{i+1}^p$ is a $p$-morphism that acts exactly on $T_{i+1}^{a,p}$. As we can do this for all $p$, we can go with commutation relations from  $(^{b_n}\hspace{-0.08cm}f,\ldots ,^{b_1}\hspace{-0.15cm}f)$ to an expression of the form $$(^{b''_v}\hspace{-0.1cm}f,\ldots ,^{b''_1}\hspace{-0.15cm}f, \omega_{i+1}^{\gamma},\ldots,\omega_{i+1}^2,\omega_{i+1}^1)$$
The fact that $( \omega_{i+1}^{\gamma},\ldots,\omega_{i+1}^2,\omega_{i+1}^1)$ is an expression of the morphism $\omega_{i+1}$ allows us to finish the proof of property ($\ast$) and of theorem \ref{huge}.

\end{proof}
\vspace{0.7cm}
\subsection{The idempotents}\label{olaila}

 We can now define a special morphism $f_{\bar{s}}=^{i_k}\hspace{-0.17cm}f\circ \cdots\circ ^{i_2}\hspace{-0.1cm}f\circ ^{i_1}\hspace{-0.1cm}f\in \mathrm{End}(\theta_{\bar{s}})$\label{K18} : it is the morphism characterized by the fact that, if $\bar{i}=(i_k,\ldots, i_1)$, then  $N_{\bar{i}}(C)\neq 0$ for all $C\in \mathrm{cores}(x)$. Theorem \ref{huge} shows that if this morphism exists, it is unique. Now we will show that at least one such morphism  exists.

If $x\in W$ we know that we can pass from any reduced expression $\bar{s}$ of $x$ to any other $\bar{t},$ by a sequence of braid relations. This induces a morphism of $f$-type in Hom($\theta_{\bar{s}},\theta_{\bar{t}}$). If we pass through all reduced expressions of $x$ in any way we want, we obtain a morphism that satisfies the requirements for $f_{\bar{s}}$. So we conclude that $f_{\bar{s}}$ is a well-defined morphism.

Theorem \ref{huge} tells us that $f_{\bar{s}}^2=f_{\bar{s}}$. This means that $f_{\bar{s}}$ is an idempotent and so we conclude that the bimodule 
$f_{\bar{s}}(\theta_{\bar{s}})$ (that we denote by $E_{\bar{s}}$\label{K19}) is an element of Soergel's category.

  \begin{theorem}\label{cototisimo}
 If $\bar{s}$ and $\bar{t}$ are two reduced expressions of the same element  $x\in W$ then $E_{\bar{s}}$ is isomorphic to  $E_{\bar{t}}$.
\end{theorem}
\begin{proof}
 We need to find an isomorphism between $E_{\bar{s}}$ and  $E_{\bar{t}}$ for any two reduced expressions $\bar{s}$ and $\bar{t}$ of the element $x\in W$.  As  we can pass from any reduced expression of $x$ to any other one by a sequence of braid relations, it is enough to find an isomorphism between $E_{\bar{s}}$ and  $E_{\bar{t}}$ when $\bar{s}$ and $\bar{t}$ differ  only by one braid relation. Let us call the associated morphism $F_{\bar{s},\bar{t}} : \theta_{\bar{s}}\rightarrow \theta_{\bar{t}}.$\label{K20} We define $i_{\bar{s}} : E_{\bar{s}} \rightarrow \theta_{\bar{s}}$, the natural inclusion. We define the projection $f'_{\bar{t}} : \theta_{\bar{t}} \rightarrow E_{\bar{t}}$ to be the same as $f_{\bar{t}}: \theta_{\bar{t}} \rightarrow \theta_{\bar{t}}$, up to the fact that we restrict the target. We define $a_{\bar{s},\bar{t}}$\label{K21} as follows :

$$\xymatrix{
{\theta_{\bar{s}}}\ar[r]^{F_{\bar{s},\bar{t}}}&{\theta_{\bar{t}}}\ar[d]^{f'_{\bar{t}}}\\
{E_{\bar{s}}}\ar[u]^{i_{\bar{s}}}\ar[r]_{a_{\bar{s},\bar{t}}}&{E_{\bar{t}}}
}
$$
To finish the proof of theorem \ref{cototisimo} we only need to prove 
$$
 a_{\bar{t},\bar{s}} \circ a_{\bar{s},\bar{t}} =\mathrm{id}_{E_{\bar{s}}},
$$
 so  we only need to prove 
\begin{equation}\label{superette}
 f_{\bar{s}} \circ  F_{\bar{t},\bar{s}}\circ f_{\bar{t}} \circ F_{\bar{s},\bar{t}} \circ  f_{\bar{s}}=f_{\bar{s}},
\end{equation}
but this is a direct consequence of theorem \ref{huge} and the definition of $f_{\bar{s}}$. \end{proof}

\begin{notation}
 If $\bar{s}$ is a reduced expression of $w\in W$ we define $E_w=E_{\bar{s}},$ and this is a well-defined bimodule modulo isomorphism. 
\end{notation}

\subsection{}
We recall that $\eta : \langle \mathbf{B}\rangle
\rightarrow \mathcal{H}$ is the inverse of  $\varepsilon.$

\begin{lemma}\label{ochoporocho}
Let $M$ be a Soergel bimodule and let the polynomials $p_w\in \mathbb{Z}[v,v^{-1}]$ be defined by $\eta(\langle M\rangle )=\sum_{w\in W}p_wT_w.$ We have the following formula : $$K\otimes_R M\cong \bigoplus_{w\in W}p_w(1) K_w, $$
where $K$ is the fraction field of $R$.  
\end{lemma}
\begin{proof}
As $(V,V)$ is a good couple (bonne paire), by  \cite[formula (3.6)]{lib} we have that $K\otimes_R \theta_s \cong K\oplus K_s. $ If $(s_1,\ldots, s_n)$ is a sequence of elements of $\mathcal{S}$ we have that $$K\otimes_R \theta_{s_1}\cdots \theta_{s_n}\cong(K\oplus K_{s_1})\otimes_R\cdots \otimes_R(K\oplus K_{s_n}).$$ 

By definition of $\varepsilon$, we have $\eta(\langle \theta_{s_1}\cdots \theta_{s_n}\rangle)=(1+T_{s_1})\cdots (1+T_{s_n}).$ To specialise this element in $q=1$ is equivalent to calculate $(1+{s_1})\cdots (1+{s_n})$ in the group algebra $k[W]$ (identifying ${s_i}$ with $T_{s_i}$), so the fact that  $K_x\otimes_R K_y\cong K_{xy}$ for all $x,y\in W, $ allows us to finish the proof of the lemma for $M=\theta_{s_1}\cdots \theta_{s_n}.$ 

It is trivial to extend this result to finite direct sums of bimodules of the form $\theta_{s_1}\cdots \theta_{s_n}$, and for the direct summands it is enough to use the characterisation of Soergel bimodules given in \cite[lemma 5.13]{S3}. \end{proof}

\begin{notation}
 In \cite{S3} Soergel classifies the indecomposable bimodules in \textbf{B} and for each $x\in W$ he defines an indecomposable bimodule $B_x$ satisfying some support properties. We define $B'_x=B_x(-l(x)).$ For our pourposes we only need to know that  $B'_x$ is indecomposable and that it is a direct summand of every product $\theta_{s_1}\cdots \theta_{s_n}$ if $s_1\cdots s_n=x.$  
\end{notation}

\begin{proposition}\label{etru}
 The indecomposable Soergel bimodule $B'_w$ is a direct summand of $E_w$ and the set $\{\eta(\langle E_w\rangle)\}_{w\in W}$ is a basis of the Hecke algebra.
\end{proposition}

\begin{proof} Let us recall that $K$ be the fraction field of $R$. Let $s,r\in \mathcal{S}$ and $ z=\underbrace{srs\cdots}_{m(s,r)\ \mathrm{times}}$. In this proof we will use the following notation : if $M$ is a Soergel bimodule, let $\overline{M}=K\otimes_R M$ be the corresponding $(K,K)$-bimodule (see \cite[lemma 3.4]{lib}). We can find a projection $\pi:X_{sr}\rightarrow B'_z$ and an injection $in: B'_z\rightarrow X_{rs}$ such that  $in\circ \pi=f_{sr}\in \mathrm{Hom}(X_{sr},X_{rs}).$

By lemma \ref{ochoporocho} we know that $K_z$ is a direct summand of $\overline{X_{sr}}, \overline{X_{rs}}$, and appears in both bimodules with multiplicity one. To see that $K_z$ is a direct summand of $\overline{B'_z}$ we also need the known fact that  Soergel's conjecture is true for dihedral groups, so $\varepsilon(v^{-n}C'_z)=\langle B'_z\rangle$ and $C'_z=v^n(\sum_{y\leq z}T_y)$ (see \cite[remark 4.4]{S3}). So we have an injection $j: K_z\hookrightarrow \overline{X_{sr}}.$ As $\mathrm{id}_K\otimes_R \pi :\overline{X_{sr}}\twoheadrightarrow \overline{B'_z}$ is a surjection and the space
$$\mathrm{Hom}_{(K,K)}(K_u,K_v)\simeq \begin{cases}
                                 \{0\}\ \mathrm{if\ } u\neq v\\
K\ \mathrm{if\ } u= v                                \end{cases}
 $$
we deduce an injection $(\mathrm{id}_K\otimes_R \pi)\circ j:K_z\hookrightarrow \overline{B'_z},$ and by composing with the injection $\mathrm{id}_K\otimes_R in$ (which is an injection because $K$ is flat over $R$), we finally  obtain an injection $ \overline{f_{sr}}\circ j : K_z \hookrightarrow \overline{X_{rs}}.$ This allows us to conclude that $K_z$ is a direct summand of $\overline{f_{sr}}(\overline{X_{sr}})=\overline{f_{sr}(X_{sr}})$. As $K_u\otimes_R K_v\cong K_{uv}$ for all $u,v\in W,$ we conclude that $K_u$ is a direct summand of $\overline{E_u},$ for all $u\in W$.


Let $\bar{s}$ be a reduced expression of $w$.
We have that $\theta_{\bar{s}}=E_w\oplus Y=B'_w\oplus X$, for some Soergel's bimodules $X,Y$, so $$\overline{\theta_{\bar{s}}}=\overline{E_w}\oplus \overline{Y}=\overline{B'_w}\oplus \overline{X} $$

By Krull-Schmidt (see \cite[remark 1.3]{S3}), $B'_w$ is either direct summand of $E_w$ or of $Y$ (recall that $B'_w$ is indecomposable). By \cite[satz 6.16]{S3} we have that $R_w$ is a submodule of $B'_w$, and as $K$ is flat over $R$, then $K_w$ is a direct summand of $\overline{B'_w}$. We have seen that $K_w$ is a direct summand of $\overline{E_w}$ and of $\overline{\theta_{\bar{s}}}$, but it has multiplicity one in $\overline{\theta_{\bar{s}}}$ so $B'_w$ is a direct summand of $E'_w$, thus proving the first part of the proposition.

By a theorem of Soergel (recalled in this paper in corollary \ref{inv} of section \ref{5}), if 
$$\eta(\langle E_w\rangle)=\sum_{v}p_vT_v,$$ then $p_v\in \mathbb{N}[v,v^{-1}]$.  As $K_w$ appears with multiplicity one in $\overline{E_w}$ we conclude that $p_w=1$, so $\eta(\langle E_w\rangle)=T_w +\sum_{v<w}p_vT_v$. This triangularity condition allows us to conclude the second part of the proposition.\end{proof}


\subsection{}If we choose for every element $x\in W$ a reduced expression $(s_1^x,\cdots,s_n^x) $ of $x$, then the following set is a basis of the Hecke algebra :
$$\underline{Y}=\{C'_{s_1^x}\cdots C'_{s_n^x} \}_{x\in W} $$

Proposition \ref{etru} makes precise the assertion that  the basis $\underline{A}=\{\eta(\langle E_w\rangle)\}_{w\in W}$ is in between the Kazhdan-Lusztig basis and the $\underline{Y}$ basis (Bott-Samelson-Demazure basis). We finish this section with an example that shows that $\underline{A}$ is different from both bases.

\textbf{Example : } Let $m(s,r)=4$. By definition $E_{srsr}$ is the image of $f_{rs}\circ f_{sr},$ and this is the indecomposable Soergel bimodule $B'_{srsr},$ so we conclude that  $\eta(\langle E_{srsr}\rangle)=v^{-4}C'_{srsr}.$
On the other hand, $E_{srs}$ is by definition the bimodule $\theta_s\theta_r\theta_s$, and so $\eta(\langle E_{srs}\rangle)=v^{-3}C'_sC'_rC'_s.$

In the next section we will get closer to the indecomposables considering some new morphisms.

\section{The bimodule $D_w$}\label{5}

In this section we will construct a bimodule $D_w$ in Soergel's category of bimodules, and we will study some of its properties. We start by recalling some of Soergel's results.

\subsection{}

\textbf{Notation.} Given a 
$\mathbb{Z}$-graded vector space  $V=\bigoplus_i V_i$, with dim$(V)<\infty$, we define its graded dimension by the formula
$$ \underline{\mathrm{dim}}V =\sum (\mathrm{dim} V_i)v^{-i}\in \mathbb{Z}[v,v^{-1}].$$

Let us recall that $R_+$ is the ideal of $R$ generated by the homogeneous elements of non zero degree. We define the graded rank of a finitely generated $\mathbb{Z}$-graded $R$-module 

$$\underline{\mathrm{rk}}M = \underline{\mathrm{dim}}(M/MR_+)\in \mathbb{Z}[v,v^{-1}]. $$

We have $\underline{\mathrm{dim}}(V(1))=v(\underline{\mathrm{dim}}V)$ and
$\underline{\mathrm{rk}}(M(1))=v(\underline{\mathrm{rk}}M).$ We define 
$\underline{\overline{\mathrm{rk}}}M$ as the image of $\underline{\mathrm{rk}}M$ under
 $v\mapsto v^{-1}$.

For $x\in W$, we define the $(R,R)-$bimodule $R_x$ as  the set $R$ with the usual left action but with the right action twisted by $x$  (in formulas : $r\cdot r'=r x(r')$ for $r\in R_x$ and $r'\in R$).

Let us recall that there exists a unique involution $d $ of the Hecke algebra with $d(v) = v^{-1}$ and $ d(T_x ) = (T_{x^{-1}})^{-1}$.
An immediate corollary of proposition 5.7, proposition 5.9 and corollary 5.16 of \cite{S3} is 

\begin{corollary}[Soergel]\label{inv}We have the following equations in the Hecke algebra\begin{displaymath}
\begin{array}{lll}
 \eta(\langle B \rangle)&=& \sum_{x\in W}
 \underline{\overline{\mathrm{rk}}}\mathrm{Hom}(B, R_x)T_x  \\
  &=& d\circ \sum_{x\in W}
 \underline{\overline{\mathrm{rk}}}\mathrm{Hom}( R_x, B)T_x
\end{array}
\end{displaymath}
\end{corollary}

\subsection{Definition of $f_{sr}^2(n)$}

In this subsection we prove that there exist a morphism generalizing $f_{rs}\circ f_{sr}$. To be more precise we need a definition.

\begin{definition}
 For $s,r\in\mathcal{S}$ and $1\leq n\leq m(s,r)$ we define $sr(n)=\underbrace{srs\cdots}_{n}$ and $$\theta_{sr}^l(n)=\underbrace{\theta_s\theta_r\theta_s\cdots}_{n}, \ \ \theta_{sr}^r(n)=\underbrace{\cdots \theta_s\theta_r\theta_s}_{n}$$ 
\end{definition}
\begin{proposition}\label{lumpen}
 For $s,r\in\mathcal{S}$ and $1\leq n\leq m(s,r)$ there exists a unique degree zero idempotent $f_{sr}^2(n)$ in End($\theta_{sr}^l(n)$) that factors through $B'_{sr(n)}$.
\end{proposition}
 
\begin{remark}
 The equation $f_{sr}^2(m(s,r))=f_{rs}\circ f_{sr}$ explains the choice of our notation.
\end{remark}

\begin{proof} To prove this proposition we only need to prove the following two assertions :

\begin{itemize}
\item the space of degree zero morphisms  in Hom($B'_{sr(n)}$\,,\,$\theta_{sr}^l(n)$) is one dimensional
\item  the space of degree zero morphisms in Hom($\theta_{sr}^l(n)$\,,\,$B'_{sr(n)}$ ) is one dimensional
\end{itemize}

Using the adjunction \cite[Lemme 3.3]{L}, this is equivalent to
\begin{enumerate}
\item the space of morphisms of degree $-2n$ in Hom($\theta_{sr}^r(n)B'_{sr(n)}$\,,\,$R$) is one dimensional
\item  the space of morphisms of degree $2n$ in Hom($R$\,,\,$\theta_{sr}^r(n)B'_{sr(n)}$ ) is one dimensional
\end{enumerate}

 We will prove the first claim. The second claim has a similar proof. 
Because of \cite[theorem 5.15]{S3}, we know that the Hom spaces between Soergel bimodules are free as right $R$-modules, so we have 
Hom($\theta_{sr}^r(n)B'_{sr(n)}$\,,\,$R$)$\cong \bigoplus_i n_i R(2i)$ for some integer numbers $n_i$. Let us see how to calculate these numbers. 
Let us define $$h_0=\underbrace{\cdots(1+T_s)(1+T_r)(1+T_s)}_{n\ \mathrm{terms}}C'_{sr(n)}v^{-n},$$
and let $\tau :\mathcal{H}\rightarrow \mathbb{Z} [v,v^{-1}]$ be the map defined by
$$\tau\left(\sum_{x\in W}p_xT_x\right) =p_1 \,\,\,\,\,\,\,\,\,\,\,\, (p_x\in \mathbb{Z} [v,v^{-1}]).$$

By corollary \ref{inv} we have the equations
\begin{displaymath}
\begin{array}{lll}
 \tau(h_0)&=& \tau \circ \eta \circ
 \varepsilon (h_0)  \\
  &=&\tau \circ \eta (\langle \theta_{sr}^r(n)B'_{sr(n)} \rangle)\\
  &=&\underline{\overline{\mathrm{rk}}}\mathrm{Hom} (\theta_{sr}^r(n)B'_{sr(n)},R)
 \\
 &=& \underline{\overline{\mathrm{rk}}} (\bigoplus_i n_i R(2i))\\
&=& \sum n_i v^{-2i} \underline{\overline{\mathrm{rk}}}R\\
&=& \sum n_i q^i.
\end{array}
\end{displaymath}

So our problem reduces to an easy problem in the Hecke algebra, namely, to prove that $\tau(h_0)=q^{n}+\sum_{i<n}n_iq^i.$
By \cite[remark 4.4]{S3} we have the formula $C'_{sr(n)}=v^n(\sum_{x\leq sr(n)}T_x)$ and by \cite[proposition 8.1.1]{GP} we know that 
\begin{equation} \tau(T_xT_{y^{-1}})=  \begin{cases} q^{l(x)}
 \text{ if } x=y \\
 0 \,\,\,\,\,\,\,\,\,\, \text{if} \,\,\, x \neq y
 \end{cases}\end{equation}
The following lemma  is easily proved by induction.
\begin{lemma}\label{Hecke} 
Let us define the polynomials $p_y\in \mathbb{Z}[q]$ by the formula $$ \underbrace{\cdots(1+T_s)(1+T_r)(1+T_s)}_{n\ \mathrm{terms}}=\sum p_yT_y. $$
If $ys<y$ then deg($p_y)\leq (n-l(y))/2$, if $ys>y$ then deg($p_y)\leq (n-l(y)-1)/2$ and  $p_{sr(n)}=1$.
\end{lemma}
This lemma allows us to finish the proof of proposition \ref{lumpen}
\end{proof}

We recall corollary 4.2 of \cite{L} :

 \begin{lemma}\label{grados} Let $(s_1,\ldots , s_n)\in \mathcal{S}^n.$
We define the integers $m_i$ by $\tau((1+T_{s_1})\cdots
(1+T_{s_n})) =\sum_im_iq^i$. Then, there exists an isomorphism of graded right 
$R-$modules 
$$\mathrm{Hom}(\theta_{s_1}\cdots \theta_{s_n} ,R) \cong \bigoplus_i m_i R(2i). $$
\end{lemma}

Using lemmas \ref{Hecke} and \ref{grados} we can conclude :
\begin{proposition}\label{gra}
If $ f \in \mathrm{Hom}(\theta_{sr}^r(n),R)$, then deg($f$)$\geq -2 [\frac{n-1}{2}],$ where $[\ ]$ stands for the floor function (the function that maps a real number to the next smallest integer).
\end{proposition}



 \subsection{}\label{utah} In this section we will explain two conjectural methods for finding $f_{sr}^2(n).$
 For this purpose  we have to introduce some morphisms ; we start by the Demazure operator $\partial_s:R\rightarrow R$, defined by $\partial_s(p)=\frac{p-s\cdot p}{2x_s}.$ We can define the following morphisms of graded $(R,R)$-bimodules :

\begin{displaymath}
\begin{array}{lll}
m_s : \theta_s\rightarrow R, && R\otimes_{R^s}R\ni p\otimes q \mapsto pq\\ \smallskip
  j_s: \theta_s\theta_s (2)\rightarrow \theta_s, && R\otimes_{R^s}R\otimes_{R^s}R \ni p\otimes q\otimes r \mapsto p\partial_s(q)\otimes r \in R\otimes_{R^s}R\\ \smallskip
\alpha_s:R\rightarrow \theta_s\theta_s(2) && 1\mapsto x_s\otimes 1\otimes 1+1\otimes 1\otimes x_s \in R\otimes_{R^s}R\otimes_{R^s}R 
\end{array}
\end{displaymath}
It is an easy consequence of the construction of the light leaves basis (LLB) in \cite{L} the fact that $f_{rs}\circ f_{sr}$ can be written as a linear combination of morphisms of the form $m_s,m_r,j_s,j_r,\alpha_s$ and $\alpha_r$ (eventually tensored, of course,  by the identity). For example, if $m(s,r)=2$ then $f_{rs}\circ f_{sr}=$id and if $m(s,r)=3$ then 
$$f_{rs}\circ f_{sr}=\mathrm{id}-\frac{1}{2\partial_s(x_r)} (\mathrm{id}\otimes m_r\otimes \mathrm{id}^2) \circ(\mathrm{id}\otimes  \alpha_r\otimes \mathrm{id})\circ(\mathrm{id}\otimes j_s)\circ (\alpha_s\otimes\mathrm{id})\circ j_s \circ (\mathrm{id}\otimes m_r\otimes \mathrm{id}) $$

\begin{conjecture} Let us fix an integer $n\in\mathbb{N}$. \begin{itemize}
\item Let us consider a Coxeter system $(W,\mathcal{S})$ with $s,r\in \mathcal{S}$ and $n=m(s,r)$. 
There exists a  set $\{\lambda_i^{s,r}\}_{i\in I}$ of rational functions in two variables and another  set $\{g_i^{s,r}\}_{i\in I}$ with each $g_i^{s,r}$ a composition of morphisms of the form $m_s,m_r,j_s,j_r,\alpha_s$ and $\alpha_r$, possibly tensored  by the identity, satisfying the equation \begin{equation}\label{trobon}f_{rs}\circ f_{sr}=\sum_{i\in I}\lambda_i^{s,r}(\partial_s(x_r),\partial_r(x_s)) g_i^{s,r},\end{equation} for all Soergel categories having $W$ as the underlying group. This means that  formula \ref{trobon} is true for any reflection faithful representation $V$ of $W$. We remark that the two scalars  $\partial_s(x_r)$ and $\partial_r(x_s)$ depend on $V$.
\item Let us consider a Coxeter system $(W',\mathcal{S'})$ with $s',r'\in \mathcal{S'}$ and $n<m(s',r')$. In any Soergel category associated to this group we have the formula $$f_{s'r'}^2(n)=\sum_{i\in I}\lambda_i^{s',r'}(\partial_{s'}(x_{r'}),\partial_r(x_{s'})) g_i^{s',r'}.$$
\end{itemize}
\end{conjecture}
 
 \begin{remark}
 This conjecture is true for $n=2$ and $n=3$. In fact $f_{sr}^2(2)=$id and $$f_{sr}^2(3)= \mathrm{id}-\frac{1}{2\partial_s(x_r)} (\mathrm{id}\otimes m_r\otimes \mathrm{id}^2) \circ(\mathrm{id}\otimes  \alpha_r\otimes \mathrm{id})\circ(\mathrm{id}\otimes j_s)\circ (\alpha_s\otimes\mathrm{id})\circ j_s \circ (\mathrm{id}\otimes m_r\otimes \mathrm{id}) $$
\end{remark}

%
%

As we explained in the introduction, we believe that $f_{sr}^2(n)$ is the composition of two maps $f_{sr}(n)\in A_n(s,r)$ and $f_{rs}(n)\in A_n(r,s)$, that are not bimodule morphisms, but only left $R$-module morphisms. This gives us our second conjectural method for finding $f_{sr}^2(n).$

\begin{conjecture}Let us fix an integer $n\in\mathbb{N}$. \begin{itemize}
\item Let us consider a Coxeter system $(W,\mathcal{S})$ with $s,r\in \mathcal{S}$ and $n=m(s,r)$. 
There exist   sets $\{\mu_i^{s,r,w}\}_{i\in I,w\leq sr(n)}$ and $\{\nu_i^{s,r,w}\}_{j\in J,w\leq sr(n)}$ of rational functions in two variables  and  sets of polynomials in two variables $\{P_i^{s,r,w}\}_{i\in I,w\leq sr(n)}$ and $\{Q_i^{s,r,w}\}_{j\in J,w\leq sr(n)}$, satisfying the equations \begin{equation}\label{aaa2}\partial_w(d)=\sum_{i\in I}\mu_i^{s,r,w}(\partial_s(x_r),\partial_r(x_s)) P_i^{s,r,w}(x_s,x_r)\hspace{0.8cm}\mathrm{for\ all\ }w\leq sr(n),\end{equation} \begin{equation}\label{aaa3}\partial_w(d)^*=\sum_{j\in J}\nu_i^{s,r,w}(\partial_s(x_r),\partial_r(x_s)) Q_i^{s,r,w}(x_s,x_r)\hspace{0.8cm}\mathrm{for\ all\ }w\leq sr(n)\end{equation} for all Soergel categories having $W$ as the underlying group.
\item Let us consider a Coxeter system $(W',\mathcal{S'})$ with $s',r'\in \mathcal{S'}$ and $n<m(s',r')$. Let us fix any Soergel category associated to this group. We define $f_{s'r'}(n)$ as a slight modification of the right-hand side of the formula of proposition \ref{calculo}. Namely, we replace $(s,r)$ in this formula by $(s',r')$, we make the sum to range over the set $\{w\leq s'r'(n)\}$  and we replace  $ \partial_w(d)$ and $\partial_w(d)^*$ by the right-hand side of equations \ref{aaa2} and \ref{aaa3} (again replacing $(s,r)$ by $(s',r')$ everywhere). Then $f_{s'r'}(n)\in A_n(s',r')$  is not an $(R,R)$-bimodule morphism, but only a left $R$-module morphism. 
In the same way we can obtain $f_{r's'}(n)\in A_n(r',s')$. We conjecture the formula $f_{s'r'}^2(n)=f_{r's'}(n)\circ f_{s'r'}(n). $
\end{itemize}
\end{conjecture}

\subsection{}

In this section we prove a generalization of corollary \ref{fsr}. 

\begin{proposition}\label{inclusion} We have the following inclusion :
 $$f_{sr}^2(n)(1\otimes_{R^s}\theta_r\theta_s\cdots)\subseteq 1\otimes_{R^s}\theta_r\theta_s\cdots$$
\end{proposition}

\begin{proof}
 We will prove a more general result, namely that any degree zero morphism $f$  in End($\theta_{sr}^l(n)$) satisfies  the property 
$$f(1\otimes_{R^s}\theta_r\theta_s\cdots)\subseteq 1\otimes_{R^s}\theta_r\theta_s\cdots \ \ \ \ \ \ \ \ \mathrm{(property\ }\ast\ast)$$

We already know that this proposition is true if $n=m(s,r)$ (see \ref{fsr}), so from now on we will suppose $n<m(s,r)$. By lemma \ref{Hecke} we have that zero is the minimal possible degree of a morphism in  End($\theta_{sr}^l(n)$). So to prove property $\ast\ast$ it is enough to prove it for the degree zero morphisms of a basis of the space End($\theta_{sr}^l(n)$) as left $R$-module.

From now on we will make heavy use of notations and results in \cite{L}. For $s\in\mathcal{S}$ we define the morphism  $$\alpha_s^i=\mathrm{id}^k\otimes \alpha_s \otimes \mathrm{id}^i \in\mathrm{Hom}(\theta_{r_1}\cdots \theta_{r_k}\theta_{t_1}\cdots \theta_{t_i}, \theta_{r_1}\cdots \theta_{r_k}\theta_s\theta_s\theta_{t_1}\cdots \theta_{t_i})$$

A slight modification of the following lemma can be found in \cite[lemme 3.3]{L}, or with a more detailed proof in \cite[lemma 5.16]{li3}. 

\begin{lemma}\label{bijection}
For $M,N \in \mathbf{B}$, the morphism 
\begin{displaymath}
\begin{array}{lll}
 \mathrm{Hom}(M\theta_s,N) &\rightarrow &  \mathrm{Hom}(M, N\theta_s )(2) \\
\ \ \ \ \ \ \ \ \ f &\mapsto &(f\otimes \mathrm{id}_{\theta_s}) \circ (\mathrm{id}_M \otimes \alpha_s)     
\end{array}
\end{displaymath}
is an isomorphism of left graded $R$-modules.
\end{lemma}

Let $L=\{f_{\bar{i}}\}_{i\in I}$ be a light leaves basis (LLB) of Hom$(\theta_{sr}^l(n)\theta_{sr}^r(n),R) $ as defined in \cite[theorem 5.1]{L} (it is a basis as left $R$-module). We have that the sequence $(s_i)_i$ defined in \cite[section 4.5]{L} is   $(s_1,s_2,\ldots, s_n,s_{n+1},\ldots,s_{2n})=(s,r,s,\ldots,t,t,\ldots,s,r,s)$, where $t=r$ if $n$ is even and $t=s$ if $n$ is odd..  We chose this LLB such that the $P(n,\bar{t})$ (see again  \cite[section 4.5]{L}) is a sequence of minimal length for all couples $(n,\bar{t})$. We will see in the sequel that there is only one LLB of Hom$(\theta_{sr}^l(n)\theta_{sr}^r(n),R) $ satisfying the preceding property, and in fact all this sequences $P(n,\bar{t})$ are trivial (they do not play any role in defining $L$).

 By lemma \ref{bijection} we deduce that the set 
$$\{(f_{\bar{i}}\otimes \mathrm{id}_{\theta^l_{sr(n)}})\circ \alpha_s^n\circ \alpha_r^{n-1}\circ \alpha_s^{n-2}\circ \cdots \circ \alpha_{t}^0 \}_{i\in I}$$
is a basis of End($\theta_{sr}^l(n)$) as left $R$-module.

Given the form of $\alpha_s$, our problem reduces to prove that  for an element $f_{\bar{i}}$ of the LLB of Hom$(\theta_{s_1}\cdots\theta_{s_{2n}},R)$  of degree $-2n$, we have 
\begin{equation}\label{coco}f_{\bar{i}}(1\otimes_{R^s}\theta_{s_2}\theta_{s_3}\cdots\theta_{s_{2n-1}}\otimes_{R^s}1)\subseteq R^s\hspace{3cm}\end{equation}

We will follow the notation of \cite[pp.17]{L} with the only exceptions that the morphisms $m^s:\theta_s\rightarrow R$ and $i_1^s:\theta_s\theta_s\rightarrow \theta_s$ defined in \cite[pp. 12]{L} will be called here $m_s$ and $j_s$ in concordance with the paper \cite{li3} and with section \ref{utah}  . We will  prove the following lemma. We put,  $\bar{i}=(i_1,\ldots,i_{2n})\in\{0,1\}^{2n}$, and $$f_{\bar{i}}=f^{j_{2n}}_{i_{2n},2n}\circ \cdots \circ f^{j_{1}}_{i_{1},1}(\mathrm{id}).$$
We note $f_{\bar{i}}^k=f^{j_{k}}_{i_{k},k}\circ \cdots \circ f^{j_{1}}_{i_{1},1}(\mathrm{id}).$

If $M\in \mathbf{B} $ and $a\in\mathrm{Hom}(M,\theta_{t_1}\cdots \theta_{t_k})$ we note $\varpi(a)=k$. If $M\in \mathbf{B} $ and $a\in\mathrm{Hom}(M,R)$, we note $\varpi(a)=0.$

\begin{lemma}
\begin{enumerate}
 \item  $f_{\bar{i}}$ is a composition of morphisms of the form $m_s, m_r, j_s$ and $j_r$ tensored on both sides by the identity.
\item For all $1\leq k< 2n$,  $\varpi(f_{\bar{i}}^k)\neq 0.$
\end{enumerate}
\end{lemma}

\begin{proof}
 For all $0\leq i,j\leq 1$,$\ $ $1\leq l \leq 2n$ and $a\in \mathrm{Hom}(\theta_{s_1}\cdots \theta_{s_{l-1}},\theta_{t_1}\cdots \theta_{t_k}) $ the definition of $f^j_{i,l}(a)$ gives the inequality \begin{equation}\label{chronopost}\vert\varpi(f^j_{i,l}(a))-\varpi(a)\vert\leq 1.\end{equation}
So if (1) is not true, we have that for some $1\leq k \leq 2n$, there is an $f_{sr}$ or an $f_{rs}$ tensored on both sides by the identity in the composition of morphisms defining $f^{j_{k}}_{i_{k},k}$. This would mean that $\varpi(f^{k-1}_{\bar{i}})= m(s,r)$, but this equation together with (\ref{chronopost}) gives that $k\geq m(s,r)$. Then $2n-k\leq 2n-m(s,r)<m(s,r)$, this last inequality because we have supposed $n<m(s,r).$ Finally $2n-k<m(s,r)$ and the equation (\ref{chronopost}) contradict the fact that $\varpi(f_{\bar{i}})=\varpi(f_{\bar{i}}^{2n})=0.$ So we have proved (1).

Now we prove (2) by reduction to absurd. Let $1\leq k< 2n$ be such that   $\varpi(f_{\bar{i}}^k)= 0.$ We assume  that $1\leq k\leq n$ ; the case $n\leq k\leq 2n$ has a similar proof.
By the first part of the lemma  we know that $k$ must satisfy  $s_k=s$, so $k$ is odd, say $k=2q+1$. By proposition \ref{gra} we have that deg($f_{\bar{i}}^k$)$\geq-2q$.

By construction of the LLB, we have that $f_{\bar{i}}=f_{\bar{i}}^k\otimes g,$ where $g$ is an element of the LLB of Hom$(\theta_{s_{k+1}}\cdots\theta_{s_{2n}},R).$ By hypothesis, deg($f_{\bar{i}}$)$=-2n$, so deg($f_{\bar{i}}^k$)$+$deg($g$)$=-2n$, but because of proposition \ref{gra} we know that \begin{displaymath}
\begin{array}{lll}
\mathrm{deg}(g)&\geq &-2-2[(2n-(k+1)-1)/2] \\
&=& -2n+2q+2
\end{array}
\end{displaymath}
 so deg($f_{\bar{i}}^k$)$+$deg($g$)$\geq -2q+(-2n+2q+2)=-2n+2$, which is absurd and allows us to finish the proof of part (2).\end{proof}

We can now finish the proof of the inclusion (\ref{coco}). Because of (1) and (2) we can see by induction in $k$ that for $1\leq k<2n$ we have $f_{\bar{i}}^k(\theta_{s_1}\cdots \theta_{s_k})\subseteq \theta_{t_1}\cdots\theta_{t_p}$ with $t_1=s_1=s$.
So we have \begin{displaymath}
\begin{array}{lll}
f_{\bar{i}}&=& f^{j_{2n}}_{i_{2n},2n}\circ f_{\bar{i}}^{2n-1}\\
&=& (m_s\circ j_s) \circ (f_{\bar{i}}^{2n-1}\otimes \mathrm{id}_{\theta_s})
\end{array}
\end{displaymath}

But because of (1) and of the explicit form of $m_s, m_r, j_s$ and $j_r$ given section \ref{utah},   for $1\leq k<2n$  we have $f_{\bar{i}}^k(1\otimes_{R^s}\theta_{s_2}\cdots \theta_{s_k})\subseteq 1\otimes_{R^s}\theta_{t_2}\cdots\theta_{t_p},$ so in particular 
$f_{\bar{i}}^{2n-1}(1\otimes_{R^s}\theta_{s_2}\cdots \theta_{s_{2n-1}})\subseteq 1\otimes_{R^s}R\in \theta_s$. We conclude that

\begin{displaymath}
\begin{array}{lll}
f_{\bar{i}}(1\otimes_{R^s}\theta_{s_2}\cdots\theta_{s_{2n-1}}\otimes_{R^s}1)&\subseteq & (m_s\circ j_s)(1\otimes_{R^s}R\otimes_{R^s}1) \\
&=& \partial_s(R)\\
&=&  R^s
\end{array}
\end{displaymath}
which proves the inclusion (\ref{coco}) and thus proposition \ref{inclusion}.
 \end{proof}

\subsection{Definition and properties of $D_w$}

Let us generalise definition \ref{iriarte}.

 \begin{definition}\label{iriarten}
We define the morphism $$^i\hspace{-0.09cm}f^2(k)=\mathrm{id}^i\otimes f^2_{sr}(k)\otimes \mathrm{id}^{n-i-k}\in\mathrm{Hom}(\theta_{s_1}\cdots \theta_{s_n},\theta_{t_1}\cdots \theta_{t_n}),$$ whenever this makes sense. 
\end{definition}

\begin{notation}\label{oshoporosho} Even though there does not exist a bimodule morphism $^i\hspace{-0.09cm}f(k)$ we will use the notation $^i\hspace{-0.09cm}f^2(k)=^i\hspace{-0.15cm}f(k)\circ ^i\hspace{-0.15cm}f(k).$ 
\end{notation}

\begin{definition}
A morphism $g$  between $\theta_{s_1}\cdots \theta_{s_n}$ and $\theta_{t_1}\cdots \theta_{t_n}$ is of G$f$-type (or generalized $f$-type) if there exists a tuple $(i_1,\ldots, i_k)$ such that $$g=^{i_k}\hspace{-0.17cm}h_k\circ \cdots\circ ^{i_2}\hspace{-0.17cm}h_2\circ ^{i_1}\hspace{-0.17cm}h_1,$$ where each $h_j$ is either $f$ or $f(k)$ for some integer $k$.
\end{definition}

It is a -quite long but easy- exercise to  repeat section \ref{ole} and all section \ref{dromo}  replacing everywhere cores by ELGcores (see definition \ref{cores}) and $f$-type by G$f$-type. All definitions are straightforward in this context and all lemmas, propositions and the theorem have almost the same proofs. So we can define a special morphism of G$f$-type, G$f_{\bar{s}}=$ $^{i_k}\hspace{-0cm}h_k\circ \cdots\circ ^{i_2}\hspace{-0.17cm}h_2\circ ^{i_1}\hspace{-0.17cm}h_1\in \mathrm{End}(\theta_{\bar{s}})$, characterized by the fact that, if $\bar{i}=(i_k,\ldots, i_1)$, then  $N_{\bar{i}}(C)\neq 0$ for all $C\in \mathrm{ELGcores}(x)$. 

As in section \ref{olaila} we can show that G$f_{\bar{s}}$ exists and it is an idempotent uniquely defined by the preceding property. So if $\bar{s}$ is a reduced expression of $w$, we can define $D_{\bar{s}}$ as the  image of G$f_{\bar{s}}$ shifted by $l(w)$. In formulae : $D_{\bar{s}}=$G$f_{\bar{s}}(\theta_{\bar{s}})(l(w))$.

\begin{theorem}\label{estranho}
 \begin{enumerate}
  \item If $\bar{s}$ and $\bar{t}$ are two reduced expressions of $w\in W$ then $D_{\bar{s}}$ is isomorphic to  $D_{\bar{t}}$. We call this (well-defined up to isomorphism) bimodule $D_w$.
\item The indecomposable Soergel bimodule $B_w$ is a direct summand of $D_w$ and the set $\{\eta(\langle D_w\rangle)\}_{w\in W}$ is a basis of the Hecke algebra.
 \end{enumerate}
\end{theorem}
\begin{proof}
 The proof of part 1) is similar to that of \ref{cototisimo} and the proof of part 2) is similar to that of proposition \ref{etru}.
\end{proof}

\textbf{Some examples :} 
\begin{enumerate}
 \item  For all couples $(s,r)\in \mathcal{S}^2$ we have that   $\eta(\langle D_{x}\rangle)=C'_{x}$  if $x\leq \underbrace{srs\cdots}_{m(s,r)\ \mathrm{terms}}$ and $l(x)\neq 3.$ 
\item Let $s,r,t,u,v$ be different elements of $\mathcal{S}$ such that $m(s,r)=4$ and  $m(s,t)=4$. Then  $\eta(\langle D_{x}\rangle)=C'_{x}$  if $x= srsrtst$, but $\eta(\langle D_{x}\rangle)\neq C'_x$ if $x=srsruvsrsr.$
\end{enumerate}

 We believe that $D_w$ will categorify $C'_w$ in a bast range of cases but not in all cases (as the preceding examples show). It would be interesting to know exactly for what $w$ this is the case.

\end{document}